\documentclass{article}
\usepackage{amsmath}
\usepackage{fancyhdr}

\newtheorem{theorem}{Theorem}

\newtheorem{corollary}[theorem]{Corollary}

\newtheorem{lemma}[theorem]{Lemma}

\newtheorem{proposition}[theorem]{Proposition}
\newtheorem{remark}[theorem]{Remark}

\newenvironment{proof}[1][Proof]{\noindent\textbf{#1.} }{\ \rule{0.5em}{0.5em}}
\input{tcilatex}
\begin{document}

\date{\today}
\title{On a classification of Killing vector fields on a tangent bundle with 
$g$-natural metric.}
\author{Stanis\l aw Ewert-Krzemieniewski (Szczecin)}
\maketitle

\begin{abstract}
The tangent bundle of a Riemannian manifold $(M,g)$ with non-degenerated $g-$
natural metric $G$ that admits a Killing vector field is investigated. Using
Taylor's formula $(TM,G)$ is decomposed into four classes that are
investigated separately. The equivalence of the existence of Killing vector
field on $M$ and $TM$ is proved.

\textbf{Mathematics Subject Classification }Primary 53B20, 53C07, secondary
53B21, 55C25.

\textbf{Key words}: Riemannian manifold, tangent bundle, g - natural metric,
Killing vector field, non-degenerate metric
\end{abstract}

\section{Introduction}

Geometry of a tangent bundle goes back to 1958 when Sasaki published (\cite%
{S}). Heaving given Riemannian metric $g$ on a differentiable manifold $M,$
he constructed a Riemannian metric $G$ on the tangent bundle $TM$ of $M,$
known today as the Sasaki metric. Since then different topics of geometry of
the tangent bundle were studied by many geometers. The Killing vector fields
on $(TM,G)$ were studied in (\cite{Abbassi 2003}), (\cite{Tanno 1}) and (%
\cite{Tanno 2}) with $G$ being the Cheeger-Gromoll metric $g^{CG}$, the
complete lift $g^{c}$ and the Sasaki metric $g^{S},$ respectively, obtained
from a metric $g$ on a base manifold $M$. Similar results were obtained
independently in (\cite{P}). All these metrics belong to a class of metrics
on $TM,$ known as a $g-$natural one and constructed in (\cite{KS}), see also
(\cite{Abb 2005 b}). These metrics can be regarded as jets of a Riemannian
metric $g$ on a manifold $M$ (\cite{A}).

In the paper we develop the method by Tanno (\cite{Tanno 2}) to investigate
Killing vector fields on $TM$ with arbitrary, non-degenerated $g-$ natural
metrics. The method applies Taylor's formula to components of the vector
field that is supposed to be an infinitesimal affine transformation, in
particular an infinitesimal isometry. The infinitesimal affine
transformation is determined by the values of its components and their first
partial derivatives at a point (\cite{KN}, p. 232). It appears by applying
the Taylor's formula there are at most four "generators" of the
infinitesimal isometry: two vectors and two tensors of type $(1,1).$

The paper is organized as follows. In Chapter 2 we describe the conventions
and give basic formulas we shall need. We also give a short resum\'{e} on a
tangent bundle of a Riemannian manifold. In Chapter 3 we calculate the Lie
derivative of a $g-$natural metric $G$ on $TM$ in terms of horizontal and
vertical lifts of vector fields from $M$ to $TM.$ Furthermore, we obtain the
Lie derivative of $G$ with respect to an arbitrary vector field in terms of
an adapted frame. By applying the Taylor's formula to the Killing vector
field on a neighbourhood of the set $M\times \{0\}$ we get a series of
conditions relating components and their covariant derivatives. Finally we
prove some lemmas of a general character. It is worth mentioning that at
this level ther is a restriction on one of the generator to be non-zero.\
The further restrictions of this kind will appear later on.

In Chapter 4, making use of these conditions and lemmas we split the
non-degenerated $g-$natural metrics on $TM$ into four classes (Theorem \ref%
{Splitting theorem}). For each such class further properties are proved
separately. Moreover, a complete structure of Killing vector fields on $TM$
for some subclasses is given (Theorems \ref{Structure 1} and \ref{Structure
2}).

As a consequence of the splitting theorem and Theorem \ref{Lift prop 2} as
well, we obtain the main

\begin{theorem}
\label{mine}If the tangent bundle of a Riemannian manifold $(M,g)$, $dimM>2,$
with a $g-$ natural, non-degenerated metric $G$ admits a Killing vector
field, then there exists a Killing vector field on $M.$

Conversely, any Killing vector field $X$ on a Riemannian manifold $(M,g)$
gives rise to the Killing vector field $Z$ on its tangent bundle endowed
with a $g-$ natural metric. Precisely, $Z$ is the complete lift of $X.$
\end{theorem}

In the next section some classical lifts of some tensor fields from $(M,g)$
to $(TM,G)$ are discussed.

Finally, in the Appendix we collect some known facts and theorems that we
use throughout the paper and also prove lemmas of a general character.

Throughout the paper all manifolds under consideration are smooth and
Hausdorff ones. The metric $g$ of the base manifold $M$ is always assumed to
be Riemannian one.

The computations in local coordinates were partially carried out and checked
using MathTensor\texttrademark\ and Mathematica\registered\ software.

\section{Preliminaries}

\subsection{Conventions and basic formulas}

Let $(M,g)$ be a pseudo-Riemannian manifold of dimension $n$ with metric $g.$
The Riemann curvature tensor $R$ is defined by 
\begin{equation*}
R(X,Y)=\nabla _{X}\nabla _{Y}-\nabla _{Y}\nabla _{X}-\nabla _{\left[ X,Y%
\right] }.
\end{equation*}%
In a local coordinate neighbourhood $(U,(x^{1},...,x^{n}))$ its components
are given by 
\begin{multline*}
R(\partial _{i},\partial _{j})\partial _{k}=R(\partial _{i},\partial
_{j},\partial _{k})=R_{kji}^{r}\partial _{r}= \\
\left( \partial _{i}\Gamma _{jk}^{r}-\partial _{j}\Gamma _{ik}^{r}+\Gamma
_{is}^{r}\Gamma _{jk}^{s}-\Gamma _{js}^{r}\Gamma _{ik}^{s}\right) \partial
_{r},
\end{multline*}%
where $\partial _{k}=\frac{\partial }{\partial x^{k}}$ and$\ \Gamma
_{jk}^{r} $ are the Christoffel symbols of the Levi-Civita connection $%
\nabla .$ We have%
\begin{equation}
\partial _{l}g_{hk}=g_{hk;l}=\Gamma _{hl}^{r}g_{rk}+\Gamma _{kl}^{r}g_{rk}.
\label{Conv3}
\end{equation}%
The Ricci identity is%
\begin{equation}
\nabla _{i}\nabla _{j}X_{k}-\nabla _{j}\nabla
_{i}X_{k}=X_{k,ji}-X_{k,ij}=-X^{s}R_{skji}.  \label{Conv4}
\end{equation}%
The Lie derivative of a metric tensor $g$ is given by%
\begin{equation}
\left( L_{X}g\right) \left( Y,Z\right) =g\left( \nabla _{Y}X,Z\right)
+g\left( Y,\nabla _{Z}X\right)  \label{Conv4a}
\end{equation}%
for all vector fields $X,$ $Y,$ $Z$ on $M.$ In local coordinates $%
(U,(x^{1},...,x^{n}))$\ we get%
\begin{equation*}
\left( L_{X^{r}\partial _{r}}g\right) _{ij}=\nabla _{i}X_{j}+\nabla
_{j}X_{i},
\end{equation*}%
where $X_{k}=g_{kr}X^{r}.$

We shall need the following properties of the Lie derivative 
\begin{multline}
L_{X}\Gamma _{ji}^{h}=\nabla _{j}\nabla _{i}X^{h}+X^{r}R_{rjis}g^{sh}=
\label{Conv5} \\
\frac{1}{2}g^{hr}\left[ \nabla _{j}\left( L_{X}g_{ir}\right) +\nabla
_{i}\left( L_{X}g_{jr}\right) -\nabla _{r}\left( L_{X}g_{ji}\right) \right] .
\end{multline}%
If $L_{X}\Gamma _{ji}^{h}=0,$ then $X$ is said to be an infinitesimal affine
transformation.

The vector field $X$ is said to be the Killing vector field or infinitesimal
isometry if%
\begin{equation*}
L_{X}g=0.
\end{equation*}

For a Killing vector field $X$ we have%
\begin{equation*}
L_{X}\nabla =0,\quad L_{X}R=0,\quad L_{X}\left( \nabla R\right) =0,....
\end{equation*}%
(\cite{Y}, p. 23 and 24).

\subsection{Tangent bundle}

Let $x$ be a point of a Riemannian manifold $(M,g),$ dim$M=n,$ covered by
coordinate neighbourhoods $(U,$ $(x^{j})),$ $j=1,...,n.\ $Let $TM\ $be
tangent bundle of $M$ and $\pi :TM\longrightarrow M\ $be a natural
projection on $M.$ If $x\in U$ and $u=u^{r}\frac{\partial }{\partial x^{r}}%
_{\mid x}\in T_{x}M\ $then $(\pi ^{-1}(U),$ $((x^{r}),(u^{r})),$ $r=1,...,n,$
is a coordinate neighbourhood on $TM.$

The space $T_{(x,u)}TM$ tangent to $TM$ at $(x,u)$ splits into direct sum%
\begin{equation*}
T_{(x,u)}TM=H_{(x,u)}TM\oplus V_{(x,u)}TM.
\end{equation*}

$V_{(x,u)}TM$ is the kernel of the differential of the projection $\pi
:TM\longrightarrow M,$ i.e. 
\begin{equation*}
V_{(x,u)}TM=Ker\left( d\pi _{\mid (x,u)}\right)
\end{equation*}%
and is called the vertical subspace of $T_{(x,u)}TM.$

Let $V\subset M$ and $W\subset T_{x}M$ be open neighbourhhoods of $x$ and $0$
respectively, diffeomorphic under exponential mapping $exp_{x}:T_{x}M%
\longrightarrow M.$ Furthermore, let $S:\pi ^{-1}(V)\longrightarrow T_{x}M$
be a smooth mapping that translates every vector $Z\in $ $\pi ^{-1}(V)$ from
the point $y$ to the point $x$ in a parallel manner along the unique
geodesic connecting $y$ and $x.$ Finally, for a given $u\in T_{x}M,$ let $%
R_{-u}:T_{x}M\longrightarrow T_{x}M$ be a translation by $u,$ i.e. $%
R_{-u}(X_{x})=X_{x}-u.$ The connection map%
\begin{equation*}
K_{(x,u)}:T_{(x,u)}TM\longrightarrow T_{x}M
\end{equation*}%
of the Levi-Civita connection $\nabla $ \ is given by 
\begin{equation*}
K_{(x,u)}(Z)=d(exp_{p}\circ R_{-u}\circ S)(Z)
\end{equation*}%
for any $Z\in T_{(x,u)}TM.$

For any smooth vector field $Z:M\longrightarrow TM$ and $X_{x}\in T_{x}M$ we
have%
\begin{equation*}
K(dZ_{x}(X_{x}))=\left( \nabla _{X}Z\right) _{x}.
\end{equation*}%
Then $H_{(x,u)}TM=Ker(K_{(x,u)})$ is called the horizontal subspace of $%
T_{(x,u)}TM.$

We have isomorphisms%
\begin{equation*}
H_{(x,u)}TM\sim T_{x}M\sim V_{(x,u)}TM.
\end{equation*}%
For any vector $X\in T_{x}M$ there exist the unique vectors: $X^{h}$ given
by $d\pi (X^{h})=X$ and $X^{v}$ given by $X^{v}(df)=Xf$ for any function $f$
on $M.$ $X^{h}$ and $X^{v}$ are called the horizontal and the vertical lifts
of $X$ to the point $(x,u)\in TM$.

The vertical lift of a vector field $X$ on $M$ is a unique vector field $%
X^{v}$ on $TM$ such that at each point $(x,u)\in TM$ its value is a vertical
lift of $X_{x}$ to the point $(x,u).$ The horizontal lift of a vector field
is defined similarly.

If $((x^{j}),$ $(u^{j})),$ $i=1,...,n,$ is a local coordinate system around
the point $(x,u)\in TM$ where $u\in T_{x}M$ and $X=X^{j}\frac{\partial }{%
\partial x^{j}},$ then%
\begin{equation*}
X^{h}=X^{j}\frac{\partial }{\partial x^{j}}-u^{r}X^{s}\Gamma _{rs}^{j}\frac{%
\partial }{\partial u^{j}},\quad X^{v}=X^{j}\frac{\partial }{\partial u^{j}},
\end{equation*}%
where $\Gamma _{rs}^{j}$ are Christoffel symbols of the Levi-Civita
connection $\nabla $ on $(M,g).$ We shall write $\partial _{k}=\frac{%
\partial }{\partial x^{k}}$ and $\delta _{k}=\frac{\partial }{\partial u^{k}}%
.$ Cf. \cite{Dombr} or \cite{Gudmundsson}. See also \cite{YI}.

In the paper we shall frequently use the frame $(\partial _{k}^{h},\partial
_{l}^{v})=\left( \left( \frac{\partial }{\partial x^{k}}\right) ^{h},\left( 
\frac{\partial }{\partial x^{l}}\right) ^{v}\right) $ known as the adapted
frame.

Every metric $g$ on $M$ defines a family of metrics on $TM.$ Between them a
class of so called $g-$ natural metrics is of special interest. The
well-known Cheeger-Gromoll and Sasaki metrics are special cases of the $g-$%
natural metrics (\cite{KS}).

\begin{lemma}
\label{Lemma 8}(\cite{Abb 2005 b}, \cite{Abb 2005 c}) Let $(M,g)$ be a
Riemannian manifold and $G$ be a $g-$natural metric on $TM.$ There exist
functions $a_{j},$ $b_{j}:<0,\infty )\longrightarrow R,$ $j=1,2,3,$ such
that for every $X,$ $Y,$ $u\in T_{x}M$%
\begin{eqnarray}
G_{(x,u)}(X^{h},Y^{h})
&=&(a_{1}+a_{3})(r^{2})g_{x}(X,Y)+(b_{1}+b_{3})(r^{2})g_{x}(X,u)g_{x}(Y,u), 
\notag \\
G_{(x,u)}(X^{h},Y^{v})
&=&a_{2}(r^{2})g_{x}(X,Y)+b_{2}(r^{2})g_{x}(X,u)g_{x}(Y,u),  \label{g1a} \\
G_{(x,u)}(X^{v},Y^{h})
&=&a_{2}(r^{2})g_{x}(X,Y)+b_{2}(r^{2})g_{x}(X,u)g_{x}(Y,u),  \notag \\
G_{(x,u)}(X^{v},Y^{v})
&=&a_{1}(r^{2})g_{x}(X,Y)+b_{1}(r^{2})g_{x}(X,u)g_{x}(Y,u),  \notag
\end{eqnarray}%
where $r^{2}=g_{x}(u,u).$ For $\dim M=1$ the same holds for $b_{j}=0,$ $%
j=1,2,3.$
\end{lemma}

Setting $a_{1}=1,$ $a_{2}=a_{3}=b_{j}=0$ we obtain the Sasaki metric, while
setting $a_{1}=b_{1}=\frac{1}{1+r^{2}},$ $a_{2}=b_{2}=0=0,$ $a_{1}+a_{3}=1,$ 
$b_{1}+b_{3}=1$ we get the Cheeger-Gromoll one.

Following (\cite{Abb 2005 b}) we put

\begin{enumerate}
\item $a(t)=a_{1}(t)\left( a_{1}(t)+a_{3}(t)\right) -a_{2}^{2}(t),$

\item $F_{j}(t)=a_{j}(t)+tb_{j}(t),$

\item $F(t)=F_{1}(t)\left[ F_{1}(t)+F_{3}(t)\right] -F_{2}^{2}(t)$

for all $t\in <0,\infty ).$
\end{enumerate}

We shall often abbreviate: $A=a_{1}+a_{3},$ $B=b_{1}+b_{3}.$

\begin{lemma}
\label{Lemma 9}(\cite{Abb 2005 b}, Proposition 2.7) The necessary and
sufficient conditions for a $g-$ natural metric $G$ on the tangent bundle of
a Riemannian manifold $(M,g)$ to be non-degenerate are $a(t)\neq 0$ and $%
F(t)\neq 0$ for all $t\in <0,\infty ).$ If $\dim M=1$ this is equivalent to $%
a(t)\neq 0$ for all $t\in <0,\infty ).$
\end{lemma}

\begin{lemma}
The Lie brackets of vector fields on the tangent bundle of the
pseudo-Riemannian manifold $M$ are given by%
\begin{eqnarray*}
\left[ X^{h},Y^{h}\right] _{\left( x,u\right) } &=&\left[ X,Y\right]
_{\left( x,u\right) }^{h}-v\left\{ R\left( X_{x},Y_{x}\right) u\right\} , \\
\left[ X^{h},Y^{v}\right] _{\left( x,u\right) } &=&\left( \nabla
_{X}Y\right) _{\left( x,u\right) }^{v}=\left( \nabla _{Y}X\right) _{\left(
x,u\right) }^{v}+\left[ X,Y\right] _{\left( x,u\right) }^{v}, \\
\left[ X^{v},Y^{v}\right] _{\left( x,u\right) } &=&0
\end{eqnarray*}%
for all vector fields $X,$ $Y$ on $M.$
\end{lemma}

\section{Killing vector field}

\subsection{Lie derivative}

Applying the formula (\ref{Conv4a}) to the $g-$natural metric $G$ on $TM$
and vertical and horizontal lifts of vector fields $X,$ $Y,$ $Z$ on $M,$
using Proposition \ref{Connection}, we get%
\begin{multline*}
\left( L_{X^{v}}G\right) \left( Y^{v},Z^{v}\right) = \\
b_{1}g(X,Z)g(Y,u)+b_{1}g(X,Y)g(Z,u)+ \\
2a_{1}^{\prime }g(Y,Z)g(X,u)+2b_{1}^{\prime }g(X,u)g(Y,u)g(Z,u),
\end{multline*}%
\begin{equation*}
\left( L_{X^{h}}G\right) \left( Y^{v},Z^{v}\right) =0,
\end{equation*}%
whence%
\begin{multline*}
\left( L_{H^{a}\partial _{a}^{h}+V^{a}\partial _{a}^{v}}G\right) \left(
\partial _{k}^{v},\partial _{l}^{v}\right) =V^{a}\left( L_{\partial
_{a}^{v}}G\right) \left( \partial _{k}^{v},\partial _{l}^{v}\right)
+\partial _{k}^{v}H^{a}G\left( \partial _{a}^{h},\partial _{l}^{v}\right) +
\\
\partial _{k}^{v}V^{a}G\left( \partial _{a}^{v},\partial _{l}^{v}\right)
+\partial _{l}^{v}H^{a}G\left( \partial _{k}^{v},\partial _{a}^{h}\right)
+\partial _{l}^{v}V^{a}G\left( \partial _{k}^{v},\partial _{a}^{v}\right) .
\end{multline*}%
Next we find%
\begin{multline*}
\left( L_{X^{h}}G\right) \left( Y^{v},Z^{h}\right) = \\
-a_{1}R(Y,u,X,Z)+a_{2}g(\nabla _{Z}X,Y)+b_{2}g(\nabla _{Z}X,u)g(Y,u),
\end{multline*}

\begin{multline*}
\left( L_{X^{v}}G\right) \left( Y^{v},Z^{h}\right) =a_{1}g(\nabla
_{Z}X,Y)+b_{1}g(\nabla _{Z}X,u)g(Y,u)+ \\
b_{2}\left[ g(X,Z)g(Y,u)+g(X,Y)g(Z,u)\right] + \\
2a_{2}^{\prime }g(Y,Z)g(X,u)+2b_{2}^{\prime }g(X,u)g(Y,u)g(Z,u),
\end{multline*}%
whence%
\begin{multline*}
\left( L_{H^{a}\partial _{a}^{h}+V^{a}\partial _{a}^{v}}G\right) \left(
\partial _{k}^{v},\partial _{l}^{h}\right) = \\
H^{a}\left( L_{\partial _{a}^{h}}G\right) \left( \partial _{k}^{v},\partial
_{l}^{h}\right) +V^{a}\left( L_{\partial _{a}^{v}}G\right) \left( \partial
_{k}^{v},\partial _{l}^{h}\right) +\partial _{k}^{v}H^{a}G\left( \partial
_{a}^{h},\partial _{l}^{h}\right) + \\
\partial _{k}^{v}V^{a}G\left( \partial _{a}^{v},\partial _{l}^{h}\right)
+\partial _{l}^{h}H^{a}G\left( \partial _{k}^{v},\partial _{a}^{h}\right)
+\partial _{l}^{h}V^{a}G\left( \partial _{k}^{v},\partial _{a}^{v}\right) .
\end{multline*}%
Finally, we have%
\begin{multline*}
\left( L_{X^{h}}G\right) \left( Y^{h},Z^{h}\right) =A\left[ g(\nabla
_{Z}X,Y)+g(\nabla _{Y}X,Z)\right] + \\
B\left[ g(\nabla _{Z}X,u)g(Y,u)+g(\nabla _{Y}X,u)g(Z,u)\right] - \\
a_{2}\left[ R(Y,u,X,Z)+R(Z,u,X,Y)\right] ,
\end{multline*}%
\begin{multline*}
\left( L_{X^{v}}G\right) \left( Y^{h},Z^{h}\right) =a_{2}\left[ g(\nabla
_{Z}X,Y)+g(\nabla _{Y}X,Z)\right] + \\
b_{2}\left[ g(\nabla _{Z}X,u)g(Y,u)+g(\nabla _{Y}X,u)g(Z,u)\right] + \\
B\left[ g(X,Y)g(Z,u)+g(X,Z)g(Y,u)\right] + \\
2A^{\prime }g(Y,Z)g(X,u)+2B^{\prime }g(X,u)g(Y,u)g(Z,u),
\end{multline*}%
whence%
\begin{multline*}
\left( L_{H^{a}\partial _{a}^{h}+V^{a}\partial _{a}^{v}}G\right) \left(
\partial _{k}^{h},\partial _{l}^{h}\right) = \\
H^{a}\left( L_{\partial _{a}^{h}}G\right) \left( \partial _{k}^{h},\partial
_{l}^{h}\right) +V^{a}\left( L_{\partial _{a}^{v}}G\right) \left( \partial
_{k}^{h},\partial _{l}^{h}\right) +\partial _{k}^{h}H^{a}G\left( \partial
_{a}^{h},\partial _{l}^{h}\right) + \\
\partial _{k}^{h}V^{a}G\left( \partial _{a}^{v},\partial _{l}^{h}\right)
+\partial _{l}^{h}H^{a}G\left( \partial _{k}^{h},\partial _{a}^{h}\right)
+\partial _{l}^{h}V^{a}G\left( \partial _{k}^{h},\partial _{a}^{v}\right) .
\end{multline*}

Suppose now that%
\begin{equation*}
Z=Z^{a}\partial _{a}+\widetilde{Z}^{\alpha }\delta _{\alpha }=Z^{a}\partial
_{a}^{h}+(\widetilde{Z}^{\alpha }+Z^{a}u^{r}\Gamma _{ar}^{\alpha })\partial
_{\alpha }^{v}=H^{a}\partial _{a}^{h}+V^{\alpha }\partial _{\alpha }^{v}
\end{equation*}%
is a vector field on $TM,$ $H^{a},$ $V^{a}$ being the horizontal and
vertical components of the vector field $Z$ on $TM$ respectively.

\begin{lemma}
\label{Lie Deriv}Let $G$ be a metric of the form (\ref{g1a}) defined on the
tangent bundle $TM$ of a manifold $(M,g).$ In particular, $G$ being the $g$
- natural metric on $TM.$ With respect to the base $\left( \partial
_{k}^{v},\partial _{l}^{h}\right) $ we have%
\begin{multline}
\left( L_{H^{a}\partial _{a}^{h}+V^{\alpha }\partial _{\alpha }^{v}}G\right)
\left( \partial _{k}^{h},\partial _{l}^{h}\right) =  \label{LD10} \\
-a_{2}\left[ R_{aklr}+R_{alkr}\right] H^{a}u^{r}+ \\
A\left[ \left( \partial _{k}^{h}H^{a}+H^{r}\Gamma _{rk}^{a}\right)
g_{al}+\left( \partial _{l}^{h}H^{a}+H^{r}\Gamma _{rl}^{a}\right) g_{ak}%
\right] + \\
B\left[ \left( \partial _{k}^{h}H^{a}+H^{r}\Gamma _{rk}^{a}\right)
u_{a}u_{l}+\left( \partial _{l}^{h}H^{a}+H^{r}\Gamma _{rl}^{a}\right)
u_{a}u_{k}\right] + \\
a_{2}\left[ \left( \partial _{k}^{h}V^{a}+V^{r}\Gamma _{rk}^{a}\right)
g_{al}+\left( \partial _{l}^{h}V^{a}+V^{r}\Gamma _{rl}^{a}\right) g_{ak}%
\right] + \\
b_{2}\left[ \left( \partial _{k}^{h}V^{a}+V^{r}\Gamma _{rk}^{a}\right)
u_{a}u_{l}+\left( \partial _{l}^{h}V^{a}+V^{r}\Gamma _{rl}^{a}\right)
u_{a}u_{k}\right] + \\
2A^{\prime }g_{kl}V^{b}u_{b}+2B^{\prime }V^{a}u_{a}u_{k}u_{l}+B\left(
V_{k}u_{l}+V_{l}u_{k}\right) ,
\end{multline}%
\begin{multline}
\left( L_{H^{a}\partial _{a}^{h}+V^{\alpha }\partial _{\alpha }^{v}}G\right)
\left( \partial _{k}^{v},\partial _{l}^{h}\right) =  \label{LD11} \\
-a_{1}R_{alkr}u^{r}H^{a}+\partial _{k}^{v}H^{a}\left(
Ag_{al}+Bu_{a}u_{l}\right) + \\
a_{2}\left( \partial _{l}^{h}H^{a}+H^{r}\Gamma _{rl}^{a}\right)
g_{ak}+b_{2}\left( \partial _{l}^{h}H^{a}+H^{r}\Gamma _{rl}^{a}\right)
u_{a}u_{k}+ \\
\partial _{k}^{v}V^{a}\left( a_{2}g_{al}+b_{2}u_{a}u_{l}\right) + \\
a_{1}\left( \partial _{l}^{h}V^{a}+V^{r}\Gamma _{rl}^{a}\right)
g_{ak}+b_{1}\left( \partial _{l}^{h}V^{a}+V^{r}\Gamma _{rl}^{a}\right)
u_{a}u_{k}+ \\
2a_{2}^{\prime }g_{kl}V^{b}u_{b}+2b_{2}^{\prime
}V^{a}u_{a}u_{k}u_{l}+b_{2}\left( V_{k}u_{l}+V_{l}u_{k}\right) ,
\end{multline}%
\begin{multline}
\left( L_{H^{a}\partial _{a}^{h}+V^{\alpha }\partial _{\alpha }^{v}}G\right)
\left( \partial _{k}^{v},\partial _{l}^{v}\right) =  \label{LD12} \\
a_{2}\left( \partial _{k}^{v}H^{a}g_{al}+\partial _{l}^{v}H^{a}g_{ak}\right)
+b_{2}\left( \partial _{k}^{v}H^{a}u_{a}u_{l}+\partial
_{l}^{v}H^{a}u_{a}u_{k}\right) + \\
b_{1}\left( V_{k}u_{l}+V_{l}u_{k}\right) +2a_{1}^{\prime
}g_{kl}V^{b}u_{b}+2b_{1}^{\prime }V^{b}u_{b}u_{k}u_{l}+ \\
a_{1}\left( \partial _{k}^{v}V^{a}g_{al}+\partial _{l}^{v}V^{a}g_{ak}\right)
+b_{1}\left( \partial _{k}^{v}V^{a}u_{a}u_{l}+\partial
_{l}^{v}V^{a}u_{a}u_{k}\right) .
\end{multline}
\end{lemma}

\subsection{Taylor's formula and coefficients}

Throughout the paper the following hypothesis will be used:%
\begin{eqnarray}
&&(M,g)\text{ is a Riemannian manifold of dimension }n\text{ with metric }g,%
\text{ }  \TCItag{H}  \label{H} \\
&&\text{covered by the coordinate system }(U,\text{ }(x^{r})).  \notag \\
&&(TM,G)\text{ is the tangent bundle of }M\text{ with }g\text{-natural non-}
\notag \\
&&\text{degenerated metric }G,\text{ covered by a coordinate system }  \notag
\\
&&(\pi ^{-1}(U),\text{ }(x^{r},u^{s}))\text{, }r,s\text{ run through the
range }\{1,...,n\}.  \notag \\
&&Z\text{ is a Killing vector field on }TM\text{ with local components }%
(Z^{r},\widetilde{Z}^{s})\text{ }  \notag \\
&&\text{with respect to the local base }(\partial _{r},\delta _{s})\text{ }.
\notag
\end{eqnarray}

Let%
\begin{multline}
H^{a}=Z^{a}=Z^{a}(x,u)=  \label{Taylor1} \\
X^{a}+K_{p}^{a}u^{p}+\frac{1}{2}E_{pq}^{a}u^{p}u^{q}+\frac{1}{3!}%
F_{pqr}^{a}u^{p}u^{q}u^{r}+\frac{1}{4!}G_{pqrs}^{a}u^{p}u^{q}u^{r}u^{s}+%
\cdots ,
\end{multline}%
\begin{multline}
\widetilde{Z}^{a}=\widetilde{Z}^{a}(x,u)=  \label{Taylor2} \\
Y^{a}+\widetilde{P}_{p}^{a}u^{p}+\frac{1}{2}Q_{pq}^{a}u^{p}u^{q}+\frac{1}{3!}%
S_{pqr}^{a}u^{p}u^{q}u^{r}+\frac{1}{4!}V_{pqrs}^{a}u^{p}u^{q}u^{r}u^{s}+%
\cdots
\end{multline}%
be expansions of the components $Z^{a}$ and $\widetilde{Z}^{a}$ by Taylor's
formula in a neighbourhood of \ a point $(x,0)\in TM.$ For each $a$ the
coefficients are values of partial derivatives of $Z^{a}$, $\widetilde{Z}%
^{a} $ respectively$,$ taken at a point $(x,0)$ and therefore are symmetric
in all lower indices. For simplicity we have omitted the remainders.

\begin{lemma}
(\cite{Tanno 2}) The quantities%
\begin{multline*}
X=\left( X^{a}(x)\right) =\left( Z^{a}(x,0)\right) , \\
Y=\left( Y^{a}\left( x\right) \right) =\left( \widetilde{Z}^{a}\left(
x,0\right) \right) , \\
K=\left( K_{p}^{a}\left( x\right) \right) =\left( \delta _{p}Z^{a}\left(
x,0\right) \right) , \\
E=\left( E_{pq}^{a}\left( x\right) \right) =\left( \delta _{p}\delta
_{q}Z^{a}\left( x,0\right) \right) , \\
P=\left( P_{p}^{a}\left( x\right) \right) =\left( \left( \delta _{p}%
\widetilde{Z}^{a}\right) \left( x,0\right) -\partial _{p}\left( Z^{a}\left(
x,0\right) \right) \right)
\end{multline*}%
are tensor fields $M.$
\end{lemma}

Applying the operators $\partial _{k}^{v}$ and $\partial _{k}^{h}$ to the
horizontal components we get%
\begin{equation*}
\partial _{k}^{v}H^{a}=K_{k}^{a}+E_{kq}^{a}u^{q}+\frac{1}{2}%
F_{kpq}^{a}u^{p}u^{q}+\frac{1}{3!}G_{kpqr}^{a}u^{p}u^{q}u^{r}+\cdots ,
\end{equation*}%
\begin{multline*}
\partial _{k}^{h}H^{a}=\Theta _{k}X^{a}+\Theta _{k}K_{p}^{a}u^{p}+ \\
\frac{1}{2}\Theta _{k}E_{pq}^{a}u^{p}u^{q}+\frac{1}{3!}\Theta
_{k}F_{pqr}^{a}u^{p}u^{q}u^{r}+\frac{1}{4!}\Theta
_{k}G_{pqrs}^{a}u^{p}u^{q}u^{r}u^{s}+\cdots .
\end{multline*}%
on a neighbourhood of a point $(x,0)\in TM,$ where for any $(1,z)$ tensor $T$
we have put 
\begin{equation*}
\Theta _{k}T_{hij...}^{a}=\nabla _{k}T_{hij...}^{a}-\Gamma
_{rk}^{a}T_{hij....}^{r}.
\end{equation*}%
Moreover, if we put:%
\begin{equation*}
S_{k}^{a}=\widetilde{P}_{k}^{a}+X^{b}\Gamma _{bk}^{a}=\widetilde{P}%
_{k}^{a}-\partial _{k}X^{a}+\nabla _{k}X^{a}=P_{k}^{a}+\nabla _{k}X^{a},
\end{equation*}%
\begin{equation*}
T_{kp}^{a}=Q_{kq}^{a}+K_{k}^{b}\Gamma _{bp}^{a}+K_{p}^{b}\Gamma _{bk}^{a},
\end{equation*}%
\begin{eqnarray*}
F_{lkpq} &=&\delta _{k}\delta _{p}\delta _{q}Z^{a}(x,0)g_{al},\quad \\
W_{lkpq} &=&\left( \delta _{k}\delta _{p}\delta _{q}\overline{Z}%
^{a}(x,0)+E_{pk}^{c}\Gamma _{cq}^{a}+E_{qk}^{c}\Gamma
_{cp}^{a}+E_{pq}^{c}\Gamma _{ck}^{a}\right) g_{al}= \\
&&\left( S_{kpq}^{a}+E_{pk}^{c}\Gamma _{cq}^{a}+E_{qk}^{c}\Gamma
_{cp}^{a}+E_{pq}^{c}\Gamma _{ck}^{a}\right) g_{al},
\end{eqnarray*}%
\begin{equation*}
Z_{kpqr}^{a}=V_{kpqr}^{a}+F_{kpq}^{c}\Gamma _{cr}^{a}+F_{kqr}^{c}\Gamma
_{cp}^{a}+F_{krp}^{c}\Gamma _{cq}^{a}+F_{pqr}^{c}\Gamma _{ck}^{a},
\end{equation*}%
then the vertical component writes%
\begin{equation*}
V^{a}=Y^{a}+S_{p}^{a}u^{p}+\frac{1}{2!}T_{pq}^{a}u^{p}u^{q}+\frac{1}{3!}%
W_{pqr}^{a}u^{p}u^{q}u^{r}+\frac{1}{4!}Z_{pqrs}^{a}u^{p}u^{q}u^{r}u^{s}+%
\cdots
\end{equation*}%
and%
\begin{equation*}
\partial _{k}^{v}V^{a}=S_{k}^{a}+T_{kp}^{a}u^{p}+\frac{1}{2}%
W_{kpq}^{a}u^{p}u^{q}+\frac{1}{3!}Z_{kpqr}^{a}u^{p}u^{q}u^{r}+...,
\end{equation*}%
\begin{multline}
\partial _{k}^{h}V^{a}=\Theta _{k}Y^{a}+\Theta _{k}S_{p}^{a}u^{p}+
\label{Taylor7} \\
\frac{1}{2!}\Theta _{k}T_{pq}^{a}u^{p}u^{q}+\frac{1}{3!}\Theta
_{k}W_{pqr}^{a}u^{p}u^{q}u^{r}+\frac{1}{4!}\Theta
_{k}Z_{pqrs}^{a}u^{p}u^{q}u^{r}u^{s}+\cdots
\end{multline}%
on a neighbourhood of a point \thinspace $(x,0)\in TM.$

We shall often use the following definitions and abbreviations:%
\begin{equation*}
S_{p}^{a}=P_{p}^{a}+\nabla _{p}X^{a},\quad S_{kp}=S_{p}^{a}g_{ak},\quad
P_{lk}=P_{k}^{a}g_{al},
\end{equation*}%
\begin{equation*}
K_{lp}=K_{p}^{a}g_{al},\quad E_{kpq}=E_{kqp}=E_{pq}^{a}g_{ak},\quad
T_{lkp}=T_{kp}^{a}g_{al}.
\end{equation*}

Substituting (\ref{Taylor1}) - (\ref{Taylor7}) into the right hand sides of (%
\ref{LD10})-(\ref{LD12}) we obtain on some neighbourhood of $(x,0)$
expressions that are sums of polynomials in variables $u^{r}$ with
coefficients depending on $x^{t}$ multiplied by functions depending on $%
r^{2}=g_{rs}u^{r}u^{s}$ plus terms that contain remainders. Suppose that $%
Z=Z^{r}\partial _{r}+\widetilde{Z}^{r}\delta _{r}$ is a Killing vector field
on $TM.$ Then the left hand sides vanish and substituting $u=(u^{j})=0$ we
obtain on $M\times \{0\}$%
\begin{equation}
A\left( \nabla _{k}X_{l}+\nabla _{l}X_{k}\right) +a_{2}\left( \nabla
_{k}Y_{l}+\nabla _{l}Y_{k}\right) =0,  \tag{$I_{1}$}  \label{I1}
\end{equation}%
\begin{equation}
AK_{lk}+a_{2}\left( P_{lk}+\nabla _{k}X_{l}+\nabla _{l}X_{k}\right)
+a_{1}\nabla _{l}Y_{k}=0,  \tag{$II_{1}$}  \label{II1}
\end{equation}%
\begin{equation}
a_{2}\left( K_{lk}+K_{kl}\right) +a_{1}\left( S_{lk}+S_{kl}\right) =0, 
\tag{$III_{1}$}  \label{III1}
\end{equation}%
where $A=A(0),$ $a_{j}=a_{j}(0).$ Differentiating with respect to $\delta
_{k}$, making use of the property%
\begin{equation*}
\delta _{k}f(r^{2})=2f^{\prime }(r^{2})g_{ks}u^{s}
\end{equation*}%
and substituting $u^{j}=0$ we find

\begin{multline}
A\left( \nabla _{k}K_{lp}+\nabla _{l}K_{kp}\right) +a_{2}\left[ \nabla
_{k}S_{lp}+\nabla _{l}S_{kp}-X^{a}\left( R_{aklp}+R_{alkp}\right) \right] + 
\tag{$I_{2}$}  \label{I2} \\
2A^{\prime }g_{kl}Y_{p}+B\left( Y_{k}g_{lp}+Y_{l}g_{kp}\right) =0,
\end{multline}%
\begin{multline}
AE_{lkp}+a_{1}\left( \nabla _{l}S_{kp}-X^{a}R_{alkp}\right) +a_{2}\left(
\nabla _{l}K_{kp}+T_{lkp}\right) +  \tag{$II_{2}$}  \label{II2} \\
2a_{2}^{\prime }g_{kl}Y_{p}+b_{2}\left( Y_{k}g_{lp}+Y_{l}g_{kp}\right) =0,
\end{multline}

\begin{equation}
a_{1}\left( T_{lkp}+T_{klp}\right) +a_{2}\left( E_{lkp}+E_{klp}\right)
+b_{1}\left( Y_{k}g_{lp}+Y_{l}g_{kp}\right) +2a_{1}^{\prime }g_{kl}Y_{p}=0, 
\tag{$III_{2}$}  \label{III2}
\end{equation}%
on $M\times \{0\},$ where $A^{\prime }=A^{\prime }(0),$ $a_{j}^{\prime
}=a_{j}^{\prime }(0)$ etc.

For any $(0,2)$ tensor $T$ we put%
\begin{equation*}
\overline{T}_{ab}=T_{ab}+T_{ba},\quad \widehat{T}_{ab}=T_{ab}-T_{ba}
\end{equation*}%
It is easily seen, that the quantities $F$ and $W$\ are symmetric in the
last three indices. Proceeding in the same way as before we easily obtain
expressions of the second order:%
\begin{multline}
\left( L_{H^{a}\partial _{a}^{h}+V^{\alpha }\partial _{\alpha }^{v}}G\right)
\left( \partial _{k}^{h},\partial _{l}^{h}\right) _{pq}|_{(x,0)}= 
\tag{$I_{3}$}  \label{I3} \\
A\left( \nabla _{k}E_{lpq}+\nabla _{l}E_{kpq}\right) +a_{2}\left( \nabla
_{k}T_{lpq}+\nabla _{l}T_{kpq}\right) +2A^{\prime }g_{kl}\overline{S}_{pq}+
\\
B\left[ \left( \nabla _{k}X_{p}+S_{kp}\right) g_{ql}+\left( \nabla
_{k}X_{q}+S_{kq}\right) g_{pl}+\right. \\
\left. \left( \nabla _{l}X_{p}+S_{lp}\right) g_{qk}+\left( \nabla
_{l}X_{q}+S_{lq}\right) g_{pk}\right] + \\
b_{2}\left( \nabla _{k}Y_{p}g_{ql}+\nabla _{k}Y_{q}g_{pl}+\nabla
_{l}Y_{p}g_{qk}+\nabla _{l}Y_{q}g_{pk}\right) - \\
a_{2}\left[ K_{p}^{a}\left( R_{alkq}+R_{aklq}\right) +K_{q}^{a}\left(
R_{alkp}+R_{aklp}\right) \right] =0,
\end{multline}%
\begin{multline}
\left( L_{H^{a}\partial _{a}^{h}+V^{\alpha }\partial _{\alpha }^{v}}G\right)
\left( \partial _{k}^{v},\partial _{l}^{h}\right) _{pq}|_{(x,0)}= 
\tag{$II_{3}$}  \label{II3} \\
AF_{lkpq}+a_{2}W_{lkpq}+a_{1}\nabla _{l}T_{kpq}+a_{2}\nabla
_{l}E_{kpq}+2a_{2}^{\prime }g_{kl}\overline{S}_{pq}- \\
a_{1}\left( K_{p}^{a}R_{alkq}+K_{q}^{a}R_{alkp}\right) +B\left(
K_{pk}g_{ql}+K_{qk}g_{pl}\right) + \\
b_{2}\left( \overline{S}_{pk}g_{ql}+\overline{S}%
_{qk}g_{pl}+S_{lp}g_{qk}+S_{lq}g_{pk}+\nabla _{l}X_{p}g_{kq}+\nabla
_{l}X_{q}g_{kp}\right) + \\
b_{1}\left( \nabla _{l}Y_{p}g_{kq}+\nabla _{l}Y_{q}g_{kp}\right) =0,
\end{multline}

\begin{multline}
\left( L_{H^{a}\partial _{a}^{h}+V^{\alpha }\partial _{\alpha }^{v}}G\right)
\left( \partial _{k}^{v},\partial _{l}^{v}\right) _{pq}|_{(x,0)}= 
\tag{$III_{3}$}  \label{III3} \\
a_{2}\left( F_{lkpq}+F_{klpq}\right) +a_{1}\left( W_{lkpq}+W_{klpq}\right)
+2a_{1}^{\prime }g_{kl}\overline{S}_{pq}+ \\
b_{1}\left( \overline{S}_{kp}g_{ql}+\overline{S}_{kq}g_{pl}+\overline{S}%
_{lp}g_{qk}+\overline{S}_{lq}g_{pk}\right) + \\
b_{2}\left( K_{pk}g_{ql}+K_{qk}g_{pl}+K_{pl}g_{qk}+K_{ql}g_{pk}\right) =0.
\end{multline}

Finally, expressions of the third order are:%
\begin{multline}
\left( L_{H^{a}\partial _{a}^{h}+V^{\alpha }\partial _{\alpha }^{v}}G\right)
\left( \partial _{k}^{h},\partial _{l}^{h}\right) _{pqr}|_{(x,0)}= 
\tag{$I_{4}$}  \label{I4} \\
A\left[ \nabla _{k}F_{lpqr}+\nabla _{l}F_{kpqr}\right] +a_{2}\left[ \nabla
_{k}W_{lpqr}+\nabla _{l}W_{kpqr}\right] - \\
a_{2}\left[ E_{pq}^{a}\left( R_{alkr}+R_{aklr}\right) +E_{qr}^{a}\left(
R_{alkp}+R_{aklp}\right) +E_{rp}^{a}\left( R_{alkq}+R_{aklq}\right) \right] +
\\
B\left[ \nabla _{k}\overline{K}_{qp}g_{lr}+\nabla _{k}\overline{K}%
_{rq}g_{lp}+\nabla _{k}\overline{K}_{pr}g_{lq}+\nabla _{l}\overline{K}%
_{qp}g_{kr}+\nabla _{l}\overline{K}_{rq}g_{kp}+\nabla _{l}\overline{K}%
_{pr}g_{kq}\right] + \\
b_{2}\left[ \nabla _{k}\overline{S}_{qp}g_{lr}+\nabla _{k}\overline{S}%
_{rq}g_{lp}+\nabla _{k}\overline{S}_{pr}g_{lq}+\nabla _{l}\overline{S}%
_{qp}g_{kr}+\nabla _{l}\overline{S}_{rq}g_{kp}+\nabla _{l}\overline{S}%
_{pr}g_{kq}\right] + \\
B\left[
g_{lp}T_{kqr}+g_{lq}T_{krp}+g_{lr}T_{kpq}+g_{kp}T_{lqr}+g_{kq}T_{lrp}+g_{kr}T_{lpq}%
\right] + \\
2B^{\prime }\left[ \left( g_{pk}g_{ql}+g_{qk}g_{pl}\right) Y_{r}+\left(
g_{qk}g_{rl}+g_{rk}g_{ql}\right) Y_{p}+\left(
g_{rk}g_{pl}+g_{pk}g_{rl}\right) Y_{q}\right] + \\
2A^{\prime }g_{kl}M_{pqr}=0,
\end{multline}%
\begin{multline}
\left( L_{H^{a}\partial _{a}^{h}+V^{\alpha }\partial _{\alpha }^{v}}G\right)
\left( \partial _{k}^{v},\partial _{l}^{h}\right) _{pqr}|_{(x,0)}= 
\tag{$II_{4}$}  \label{II4} \\
AG_{lkpqr}+a_{2}Z_{lkpqr}+a_{2}\nabla _{l}F_{kpqr}+a_{1}\nabla _{l}W_{kpqr}-
\\
a_{1}\left[ E_{pq}^{a}R_{alkr}+E_{qr}^{a}R_{alkp}+E_{rp}^{a}R_{alkq}\right] +
\\
b_{2}\left[ \nabla _{l}\overline{K}_{qp}g_{kr}+\nabla _{l}\overline{K}%
_{rq}g_{kp}+\nabla _{l}\overline{K}_{pr}g_{kq}\right] + \\
B\left[ g_{lr}\left( E_{qkp}+E_{pkq}\right) +g_{lp}\left(
E_{rkq}+E_{qkr}\right) +g_{lq}\left( E_{pkr}+E_{rkp}\right) \right] + \\
b_{1}\left[ \nabla _{l}\overline{S}_{qp}g_{kr}+\nabla _{l}\overline{S}%
_{rq}g_{kp}+\nabla _{l}\overline{S}_{pr}g_{kq}\right] + \\
b_{2}\left[ g_{kp}T_{lqr}+g_{kq}T_{lrp}+g_{kr}T_{lpq}\right] +b_{2}\left[
g_{lp}M_{kqr}+g_{lq}M_{krp}+g_{lr}M_{kpq}\right] + \\
2b_{2}^{\prime }\left[ \left( g_{pk}g_{ql}+g_{qk}g_{pl}\right) Y_{r}+\left(
g_{qk}g_{rl}+g_{rk}g_{ql}\right) Y_{p}+\left(
g_{rk}g_{pl}+g_{pk}g_{rl}\right) Y_{q}\right] + \\
2a_{2}^{\prime }g_{kl}M_{pqr}=0,
\end{multline}%
where $M_{pqr}=T_{pqr}+T_{qrp}+T_{rpq}$ and 
\begin{equation*}
Z_{lkpqr}=\left( V_{kpqr}^{a}+F_{kpq}^{c}\Gamma _{cr}^{a}+F_{kqr}^{c}\Gamma
_{cp}^{a}+F_{krp}^{c}\Gamma _{cq}^{a}+F_{pqr}^{c}\Gamma _{ck}^{a}\right)
g_{al}
\end{equation*}%
is symmetric in the last four lower indices.

Moreover, we have 
\begin{multline}
\left( L_{H^{a}\partial _{a}^{h}+V^{\alpha }\partial _{\alpha }^{v}}G\right)
\left( \partial _{k}^{v},\partial _{l}^{v}\right) _{pqr}|_{(x,0)}= 
\tag{$III_{4}$}  \label{III4} \\
a_{2}\left( G_{lkpqr}+G_{klpqr}\right) +a_{1}\left(
Z_{lkpqr}+Z_{klpqr}\right) + \\
b_{2}\left[ g_{lr}\left( E_{qkp}+E_{pkq}\right) +g_{lp}\left(
E_{rkq}+E_{qkr}\right) +g_{lq}\left( E_{pkr}+E_{rkp}\right) \right] + \\
b_{2}\left[ g_{kr}\left( E_{qlp}+E_{plq}\right) +g_{kp}\left(
E_{rlq}+E_{qlr}\right) +g_{kq}\left( E_{plr}+E_{rlp}\right) \right] + \\
b_{1}\left[ g_{kp}M_{lqr}+g_{kq}M_{lrp}+g_{kr}M_{lpq}\right] +b_{1}\left[
g_{lp}M_{kqr}+g_{lq}M_{krp}+g_{lr}M_{kpq}\right] + \\
2b_{1}^{\prime }\left[ \left( g_{pk}g_{ql}+g_{qk}g_{pl}\right) Y_{r}+\left(
g_{qk}g_{rl}+g_{rk}g_{ql}\right) Y_{p}+\left(
g_{rk}g_{pl}+g_{pk}g_{rl}\right) Y_{q}\right] + \\
2a_{1}^{\prime }g_{kl}M_{pqr}=0.
\end{multline}%
Let us remind once again that these identities hold on $M$ and therefor all
the coefficients $a_{j},$ $b_{j},$ $a_{j}^{\prime },$ $b_{j}^{\prime }$ are
considered to be constants.

\subsection{Lemmas}

\begin{lemma}
\label{Lemmas L5}Under hypothesis (\ref{H}) at a point $(x,0)\in TM$ we have:%
\begin{equation}
a_{1}T_{lkp}+a_{2}E_{lkp}=a_{1}^{\prime }\left(
Y_{l}g_{kp}-Y_{k}g_{lp}-Y_{p}g_{kl}\right) -b_{1}Y_{l}g_{kp},  \label{L 5}
\end{equation}%
\begin{equation}
AE_{lkp}+a_{2}T_{lkp}+a_{2}^{\prime }(g_{kl}Y_{p}+g_{pl}Y_{k})+\frac{1}{2}%
b_{2}(2g_{kp}Y_{l}+g_{lp}Y_{k}+g_{kl}Y_{p})=0.  \label{L5a}
\end{equation}%
If $a\neq 0,$ then%
\begin{multline}
aE_{lkm}=(a_{2}b_{1}-a_{1}b_{2}-a_{2}a_{1}^{\prime })g_{km}Y_{l}-
\label{L5a1} \\
\frac{1}{2}(a_{1}b_{2}-2a_{2}a_{1}^{\prime }+2a_{1}a_{2}^{\prime
})(g_{lm}Y_{k}+g_{lk}Y_{m}),
\end{multline}%
\begin{equation}
aT_{lkm}=(Aa_{1}^{\prime }+a_{2}b_{2}-Ab_{1})g_{km}Y_{l}+\frac{1}{2}%
(a_{2}b_{2}-2Aa_{1}^{\prime }+2a_{2}a_{2}^{\prime
})(g_{lm}Y_{k}+g_{lk}Y_{m}),  \label{L5a2}
\end{equation}%
\begin{equation}
aM_{lkm}=[2a_{2}(b_{2}+a_{2}^{\prime })-A(b_{1}+a_{1}^{\prime
})](g_{km}Y_{l}+g_{lk}Y_{m}+g_{ml}Y_{k}).  \label{L5a2'}
\end{equation}%
Moreover,%
\begin{multline}
a_{2}\left[ \nabla _{k}\left( \nabla _{l}X_{p}+\nabla _{p}X_{l}\right)
+\nabla _{l}\left( \nabla _{k}X_{p}+\nabla _{p}X_{k}\right) -\nabla
_{p}\left( \nabla _{l}X_{k}+\nabla _{k}X_{l}\right) \right] +  \label{L5b} \\
a_{1}\left( \nabla _{k}\nabla _{l}Y_{p}+\nabla _{l}\nabla _{k}Y_{p}\right)
=2A^{\prime }g_{kl}Y_{p}+B\left( Y_{k}g_{lp}+Y_{l}g_{kp}\right) ,
\end{multline}%
\begin{multline}
a\left( \nabla _{k}K_{lp}+\nabla _{l}K_{kp}\right)
+(a_{2}b_{2}+2a_{1}A^{\prime }-2a_{2}a_{2}^{\prime })Y_{p}g_{kl}+
\label{L5c} \\
\frac{1}{2}(-a_{2}b_{2}+2a_{1}B+2a_{2}a_{2}^{\prime
})(Y_{k}g_{lp}+Y_{l}g_{kp})=0.
\end{multline}
\end{lemma}

\begin{proof}
Alternating ($III_{2}$) in $(l,p),$ then interchanging the indices $(p,k)$
and adding the resulting equation to ($III_{2}$)$,$ we obtain (\ref{L 5}).

Differentiating covariantly (\ref{III1}) we get%
\begin{equation*}
a_{2}\left( \nabla _{k}K_{lp}+\nabla _{k}K_{pl}\right) +a_{1}\left( \nabla
_{k}S_{lp}+\nabla _{k}S_{pl}\right) =0.
\end{equation*}%
Symmetrizing (\ref{II2}) in $(k,p)$ and subtracting the resulting equation
from the above one we find (\ref{L5a}).

Now (\ref{L5a1}) and (\ref{L5a2}) result immediately from (\ref{L 5}) and (%
\ref{L5a}).

From (\ref{II1}) we easily get%
\begin{equation*}
A\nabla _{k}K_{lp}+a_{2}\left( \nabla _{k}P_{lp}+\nabla _{k}\nabla
_{p}X_{l}+\nabla _{k}\nabla _{l}X_{p}\right) +a_{1}\nabla _{k}\nabla
_{l}Y_{p}=0,
\end{equation*}%
whence, symmetrizing in $(k,l),$ subtracting from (\ref{I2}), by the use of
the Ricci identity, we obtain (\ref{L5b}).

To prove (\ref{L5c}) first we symmetrize (\ref{II2}) in $(k,l)$ and combine
it with (\ref{I2}) to obtain%
\begin{multline*}
a\left( \nabla _{k}K_{lm}+\nabla _{l}K_{km}\right) -a_{2}\left[ A\left(
E_{lkm}+E_{klm}\right) +a_{2}\left( T_{lkm}+T_{klm}\right) \right] + \\
2\left( a_{1}A^{\prime }-2a_{2}a_{2}^{\prime }\right) g_{kl}Y_{m}+\left(
a_{1}B-2a_{2}b_{2}\right) \left( g_{lm}Y_{k}+g_{km}Y_{l}\right) =0.
\end{multline*}%
On the other hand, symmetrizing (\ref{L5a}) in $(k,l)$ and subtracting from
the above we obtain (\ref{L5c}). This completes the proof.
\end{proof}

\begin{lemma}
Under hypothesis (\ref{H}) suppose $a\neq 0$ at a point $(x,0)\in TM.$ Then
we have%
\begin{multline}
2a\nabla _{l}K_{km}=a_{1}^{2}Y^{r}R_{rmkl}-a_{1}Bg_{km}Y_{l}+  \label{L6a} \\
(-a_{1}B+a_{2}b_{2}-2a_{2}a_{2}^{\prime
})g_{lm}Y_{k}+(-a_{2}b_{2}-2a_{1}A^{\prime }+2a_{2}a_{2}^{\prime
})g_{kl}Y_{m},
\end{multline}%
\begin{multline}
2a\left( \nabla _{l}S_{km}-X^{r}R_{rlkm}\right)
+a_{1}a_{2}Y^{r}R_{rmkl}-a_{2}Bg_{km}Y_{l}+  \label{L6b} \\
\left[ -a_{2}B+A\left( b_{2}-2a_{2}^{\prime }\right) \right] g_{lm}Y_{k}+%
\left[ -2a_{2}A^{\prime }-A\left( b_{2}-2a_{2}^{\prime }\right) \right]
g_{kl}Y_{m}=0
\end{multline}%
at the point.
\end{lemma}

\begin{proof}
From (\ref{II2}) we subtract (\ref{L5a}) to obtain 
\begin{equation}
a_{2}\nabla _{l}K_{km}+a_{1}\left( \nabla _{l}S_{km}-X^{r}R_{rlkm}\right)
+\left( a_{2}^{\prime }-\frac{b_{2}}{2}\right) \left(
g_{kl}Y_{m}-g_{ml}Y_{k}\right) =0.  \label{L6c}
\end{equation}%
On the other hand, interchanging in (\ref{II1}) $k$ and $m,$ differentiating
covariantly with respect to $\partial _{k},$ alternating in $(k,l)$ and
applying the Ricci identity, we find%
\begin{equation*}
A\left( \nabla _{k}K_{lm}-\nabla _{l}K_{km}\right) +a_{2}\left( \nabla
_{k}S_{lm}-\nabla _{l}S_{km}\right) +a_{2}X^{r}R_{rmkl}+a_{1}Y^{r}R_{rmkl}=0.
\end{equation*}%
Subtracting from (\ref{I2}), in virtue of the Bianchi identity, we get%
\begin{multline*}
2A\nabla _{l}K_{km}+2a_{2}\left( \nabla _{l}S_{km}-X^{r}R_{rlkm}\right) - \\
a_{1}Y^{r}R_{rmkl}+2A^{\prime }g_{kl}Y_{m}+B\left(
g_{lm}Y_{k}+g_{km}Y_{l}\right) =0.
\end{multline*}%
The last equation together with (\ref{L6c}) yields the result.
\end{proof}

\begin{lemma}
\label{LE5}Under hypothesis (\ref{H}) suppose $\dim M>2.$ Then on $M\times
\{0\}$%
\begin{equation}
T_{kl}=T_{lk}=2\left( b_{1}-a_{1}^{\prime }\right) \overline{S}_{kl}+b_{2}%
\overline{K}_{kl}=0,  \label{LE5-1}
\end{equation}%
\begin{multline}
a_{2}F_{labk}+a_{1}W_{labk}+\frac{1}{2}b_{2}\left( \widehat{K}_{kl}g_{ab}+%
\widehat{K}_{bl}g_{ak}+\widehat{K}_{al}g_{bk}+\overline{K}_{ak}g_{bl}\right)
+  \label{LE5-2} \\
b_{1}g_{bl}\overline{S}_{ak}+a_{1}^{\prime }(g_{kl}\overline{S}_{ab}+g_{al}%
\overline{S}_{bk})=0.
\end{multline}
\end{lemma}

\begin{proof}
Replacing in (\ref{III3}) the indices $(p,q)$ with $(a,b),$ alternating in $%
(a,l),$ then again in $(k,l)$ and adding to the first equation we get%
\begin{multline*}
a_{2}F_{labk}+a_{1}W_{labk}+ \\
\frac{1}{2}b_{2}\left( \widehat{K}_{kl}g_{ab}+2K_{bl}g_{ak}+\widehat{K}%
_{al}g_{bk}+\overline{K}_{ak}g_{bl}\right) + \\
b_{1}(g_{bl}\overline{S}_{ak}+g_{ak}\overline{S}_{bl})+a_{1}^{\prime
}(-g_{ak}\overline{S}_{bl}+g_{kl}\overline{S}_{ab}+g_{al}\overline{S}%
_{bk})=0.
\end{multline*}%
Alternating in $(a,b)$ we find 
\begin{equation*}
g_{bl}T_{ak}-g_{bk}T_{al}-g_{al}T_{bk}+g_{ak}T_{bl}=0,
\end{equation*}%
whence $(n-2)T_{ak}=0$ results. Then (\ref{LE5-2}) is obvious.
\end{proof}

\begin{lemma}
\label{AfterLE5}Under hypothesis (\ref{H}) suppose $\dim M>1$ and $a\neq 0.$
Then 
\begin{equation*}
(n-1)\beta Y_{l}=0
\end{equation*}%
on $M\times \{0\}$ holds, where%
\begin{equation*}
\beta =2A(b_{1}^{2}-a_{1}^{\prime 2}-a_{1}b_{1}^{\prime
})+(a_{1}b_{2}-2a_{2}b_{1})(3b_{2}+2a_{2}^{\prime })+2a_{2}\left[
2a_{1}^{\prime }(b_{2}+a_{2}^{\prime })+a_{2}b_{1}^{\prime }\right] .
\end{equation*}
\end{lemma}

\begin{proof}
First, replace in (\ref{III4}) the indices $(p,q,r)$ with $(a,b,c).$
Alternating an equationo btained in such a way in $(a,l),$ then in $(k,l)$
and adding the result to the first one, we get%
\begin{multline*}
a_{2}G_{labck}+a_{1}Z_{labck}+ \\
\frac{1}{2}b_{2}\left[
(E_{kcl}-E_{lck})g_{ab}+(E_{kbl}-E_{lbk})g_{ac}+2(E_{bcl}+E_{cbl})g_{ak}+(E_{acl}-E_{lac})g_{bk}+\right.
\\
\left. \left( E_{ack}+2E_{cak}+E_{kac}\right)
g_{bl}+(E_{abl}-E_{lba})g_{ck}+\left( E_{abk}+2E_{bak}+E_{kab}\right) g_{cl} 
\right] + \\
b_{1}(M_{bcl}g_{ak}+M_{ack}g_{bl}+M_{abk}g_{cl})+a_{1}^{\prime
}(-M_{bcl}g_{ak}+M_{bck}g_{al}+M_{abc}g_{kl})+ \\
b_{1}^{\prime }\left[ (g_{bl}g_{ck}+g_{bk}g_{cl})Y_{a}+2g_{ak}\left(
g_{cl}Y_{b}+g_{bl}Y_{c}\right) +(g_{ac}g_{bl}+g_{ab}g_{cl})Y_{k}-\right. \\
\left. (g_{ac}g_{bk}+g_{ab}g_{ck})Y_{l}\right] =0.
\end{multline*}%
Alternating in $(k,b)$ and contracting with $g^{ab}g^{kc}$ we obtain%
\begin{equation*}
b_{2}\left[ \left( n-2\right) E_{rls}+nE_{lrs}\right]
g^{rs}+(n-1)(b_{1}-a_{1}^{\prime })M_{rls}g^{rs}+(n+2)(n-1)b_{1}^{\prime
}Y_{l}=0,
\end{equation*}%
which, using (\ref{L5a1}) and (\ref{L5a2'}), yields the thesis.
\end{proof}

\begin{remark}
For the Cheeger-Gromoll metric $g^{CG}$ on $TM,$ the vector field $Y$
vanishes everywhere on $M$.
\end{remark}

\begin{lemma}
\label{LE6}Under hypothesis (\ref{H})%
\begin{multline}
3AF_{lkmn}+3a_{2}W_{lkmn}+B\left( g_{kl}\overline{K}_{mn}+g_{lm}\overline{K}%
_{kn}+g_{ln}\overline{K}_{km}\right) +  \label{LE6-1} \\
(b_{1}-a_{1}^{\prime })\left(
Y_{n,l}g_{km}+Y_{m,l}g_{kn}+Y_{k,l}g_{mn}\right) + \\
2(b_{2}+a_{2}^{\prime })\left( g_{kl}\overline{S}_{mn}+g_{lm}\overline{S}%
_{kn}+g_{ln}\overline{S}_{km}\right) + \\
2b_{2}\left[ g_{km}\left( X_{n,l}+S_{ln}\right) +g_{kn}\left(
X_{m,l}+S_{lm}\right) +g_{mn}\left( X_{k,l}+S_{lk}\right) \right] =0
\end{multline}%
is satisfied at a point $(x,0)\in TM$.
\end{lemma}

\begin{proof}
Differentiating covariantly (\ref{L 5}) and subtracting from (\ref{II3}) we
get%
\begin{multline}
AF_{lkmn}+a_{2}W_{lkmn}+B\left( g_{lm}K_{nk}+g_{ln}K_{mk}\right) +
\label{LE6-2} \\
(b_{1}-a_{1}^{\prime })\left(
Y_{n,l}g_{km}+Y_{m,l}g_{kn}-Y_{k,l}g_{mn}\right) - \\
a_{1}\left( K_{n}^{r}R_{rlkm}+K_{m}^{r}R_{rlkn}\right) +2a_{2}^{\prime
}g_{kl}\overline{S}_{mn}+ \\
b_{2}\left[ g_{km}\left( X_{n,l}+S_{ln}\right) +g_{kn}\left(
X_{m,l}+S_{lm}\right) +g_{ln}\overline{S}_{km}+g_{lm}\overline{S}_{kn}\right]
=0.
\end{multline}%
Antisymmetrizing in $(k,m)$ and symmetrizing in $(k,n)$ we have%
\begin{multline}
B\left[ g_{kl}\left( K_{mn}-2K_{nm}\right) +g_{lm}\left(
K_{kn}+K_{nk}\right) +g_{ln}\left( K_{mk}-2K_{km}\right) \right] +
\label{like 78} \\
2(b_{1}-a_{1}^{\prime })\left(
2Y_{m,l}g_{kn}-Y_{n,l}g_{km}-Y_{k,l}g_{mn}\right) +3a_{1}\left(
K_{n}^{r}R_{rlmk}+K_{k}^{r}R_{rlmn}\right) + \\
b_{2}\left[ 2g_{kn}\left( X_{m,l}+S_{lm}\right) -g_{km}\left(
X_{n,l}+S_{ln}\right) -g_{mn}\left( X_{k,l}+S_{lk}\right) \right] + \\
(b_{2}-2a_{2}^{\prime })\left( 2g_{lm}\overline{S}_{kn}-g_{ln}\overline{S}%
_{km}-g_{kl}\overline{S}_{mn}\right) .
\end{multline}%
Exchanging in (\ref{LE6-2}) times three the indices $k$ and $m$ and adding
to the last equation we obtain (\ref{LE6-1}). This completes the proof.
\end{proof}

\begin{lemma}
\label{LE6'}Under hypothesis (\ref{H}) relation%
\begin{multline}
3a_{2}\left[ E_{bc}^{p}\left( R_{pkal}+R_{lak}^{p}\right) +E_{ac}^{p}\left(
R_{pkbl}+R_{lbk}^{p}\right) +E_{ab}^{p}\left( R_{pkcl}+R_{lck}^{p}\right) %
\right] +  \label{LE6'-1} \\
6A^{\prime
}g_{kl}(T_{abc}+T_{bca}+T_{cab})+g_{bc}K_{kal}+g_{ca}K_{kbl}+g_{ab}K_{kcl}+
\\
g_{cl}L_{abk}+g_{al}L_{bck}+g_{bl}L_{cak}+g_{ck}L_{abl}+g_{ak}L_{bcl}+g_{bk}L_{cal}=0
\end{multline}%
holds on $M\times \{0\}$, where%
\begin{multline}
K_{kal}=K_{lak}=  \label{LE6'-2} \\
-2b_{2}\left( S_{ka,l}+S_{la,k}+X_{a,kl}+X_{a,lk}\right)
-(b_{1}-a_{1}^{\prime })(Y_{a,kl}+Y_{a,lk}),
\end{multline}%
\begin{equation}
L_{abk}=L_{bak}=2B\overline{K}_{ab,k}+3BT_{kab}+(b_{2}-2a_{2}^{\prime })%
\overline{S}_{ab,k}+3B^{\prime }(g_{ka}Y_{b}+g_{kb}Y_{a}).  \label{LE6'-3}
\end{equation}
\end{lemma}

\begin{proof}
To prove the lemma it is enough to differentiate covariantly (\ref{LE6-1})
and eliminate covariant derivatives of $F$ and $W$ from (\ref{I4}).
\end{proof}

\begin{lemma}
\label{LE6''}Under hypothesis (\ref{H}) suppose $\dim M>2.$ Then the relation%
\begin{multline*}
a_{1}\left[
2E_{ab}^{p}R_{plck}-E_{bk}^{p}R_{plac}+E_{bc}^{p}R_{plak}-E_{ak}^{p}R_{plbc}+E_{ac}^{p}R_{plbk}%
\right] + \\
B\left[ \left( E_{ckb}-E_{kcb}\right) g_{al}+\left( E_{cak}-E_{kac}\right)
g_{bl}+\right. \\
\left. \left( E_{abk}+E_{bak}\right) g_{cl}-\left( E_{abc}+E_{bac}\right)
g_{kl}\right] + \\
(b_{1}-a_{1}^{\prime })\left[ \nabla _{l}\overline{S}_{bc}g_{ak}-\nabla _{l}%
\overline{S}_{bk}g_{ac}\right] + \\
b_{2}\left[ \nabla _{l}\widehat{K}_{kc}g_{ab}+g_{ak}\left( \frac{3}{2}\nabla
_{l}K_{bc}+\frac{1}{2}\nabla _{l}K_{cb}\right) -g_{ac}\left( \frac{3}{2}%
\nabla _{l}K_{bk}+\frac{1}{2}\nabla _{l}K_{kb}\right) \right] + \\
b_{2}\left( \nabla _{l}K_{ac}g_{bk}-\nabla _{l}K_{ak}g_{bc}\right) + \\
\left( b_{2}-2a_{2}^{\prime }\right) \left(
M_{abk}g_{cl}-M_{abc}g_{kl}\right) +b_{2}\left[
g_{bk}T_{lac}-g_{bc}T_{lak}+g_{ak}T_{lbc}-g_{ac}T_{lbk}\right] + \\
2b_{2}^{\prime }\left[ \left( g_{bk}g_{cl}-g_{bc}g_{kl}\right) Y_{a}+\left(
g_{ak}g_{cl}-g_{ac}g_{kl}\right) Y_{b}+\right. \\
\left. \left( g_{al}g_{bk}+g_{ak}g_{bl}\right) Y_{c}-\left(
g_{al}g_{bc}+g_{ac}g_{bl}\right) Y_{k}\right] =0
\end{multline*}%
holds on $M\times \{0\}.$
\end{lemma}

\begin{proof}
Firstly, we change in (\ref{LE5-2}) the indices $(l,a,b,k)$ into $(k,a,b,c)$
and differentiate covariantly with respect to $\partial _{l}.$ Setting in (%
\ref{II4}) $(a,b,c)$ instead of $(p,q,r),$ subtracting the just obtained
equation and, finally, alternating in $(k,c)$ we get the thesis.
\end{proof}

\begin{lemma}
\label{LE8}Under hypothesis (\ref{H}) relations%
\begin{multline}
\mathbf{A}_{km}=(3a_{1}B-a_{2}b_{2})\nabla _{k}X_{m}+(-2a_{2}b_{1}+\frac{3}{2%
}a_{1}b_{2}+2a_{2}a_{1}^{\prime }-3a_{1}a_{2}^{\prime })\nabla _{k}Y_{m}+ \\
a_{2}B(K_{km}-2K_{mk})+(3a_{1}B-2a_{2}b_{2}+2a_{2}a_{2}^{\prime })S_{km}+ \\
(-a_{2}b_{2}+2a_{2}a_{2}^{\prime })S_{mk}=0,  \label{LE8a}
\end{multline}%
\begin{multline*}
\mathbf{F}_{kl}+\mathbf{B}%
_{kl}=2a_{2}b_{2}(L_{X}g)_{kl}+(4a_{2}b_{1}-3a_{1}b_{2}-4a_{2}a_{1}^{\prime
})(L_{Y}g)_{kl}+ \\
2\left( 3a_{2}b_{2}+3a_{1}A^{\prime }-4a_{2}a_{2}^{\prime }\right) \overline{%
S}_{kl}+2a_{2}B\overline{K}_{kl}=0.
\end{multline*}%
hold at a point $(x,0)\in TM.$
\end{lemma}

\begin{proof}
First, we change in (\ref{L5a}) the indices $(l,k,p)$ into $(l,m,n),$ then
differentiate covariantly with respect to $\partial _{k}$ and symmetrize in $%
(k,l).$ Next, change in (\ref{I3}) the indices $(p,q)$ into $(m,n)$ and
subtract the former equality to obtain%
\begin{multline*}
\frac{1}{2}\left( b_{2}-2a_{2}^{\prime }\right) \left(
Y_{n,l}g_{km}+Y_{m,l}g_{kn}+Y_{n,k}g_{lm}+Y_{m,k}g_{ln}\right)
-b_{2}(Y_{k,l}+Y_{l,k})g_{mn}- \\
a_{2}\left[ K_{n}^{r}\left( R_{rklm}+R_{rlkm}\right) +K_{m}^{r}\left(
R_{rkln}+R_{rlkn}\right) \right] +2A^{\prime }g_{kl}\overline{S}_{mn}+ \\
B\left[ g_{ln}\left( X_{m,k}+S_{km}\right) +g_{lm}\left(
X_{n,k}+S_{kn}\right) +\right. \\
\left. g_{kn}\left( X_{m,l}+S_{lm}\right) +g_{km}\left(
X_{n,l}+S_{ln}\right) \right] =0.
\end{multline*}%
Eliminating between (\ref{like 78}) and the last equation the terms
containing curvature tensor we obtain 
\begin{equation*}
g_{mn}\mathbf{B}_{kl}+g_{kl}\mathbf{F}_{mn}+g_{ln}\mathbf{A}_{km}+g_{kn}%
\mathbf{A}_{lm}+g_{lm}\mathbf{A}_{kn}+g_{km}\mathbf{A}_{ln}=0,
\end{equation*}%
where%
\begin{equation*}
\mathbf{F}_{mn}=2a_{2}B\overline{K}_{mn}+2(2a_{2}b_{2}+3a_{1}A^{\prime
}-4a_{2}a_{2}^{\prime })\overline{S}_{mn},
\end{equation*}%
\begin{equation*}
\mathbf{B}%
_{kl}=2a_{2}b_{2}(L_{X}g)_{kl}+(4a_{2}b_{1}-3a_{1}b_{2}-4a_{2}a_{1}^{\prime
})(L_{Y}g)_{kl}+2a_{2}b_{2}\overline{S}_{kl}.
\end{equation*}%
Now, the thesis is a simple consequence of Lemma \ref{Apen3}.
\end{proof}

\section{Classification}

To simplify further considerations put for a moment $\overline{X}=\nabla
_{k}X_{l}+\nabla _{l}X_{k},$ $\overline{Y}=\nabla _{k}Y_{l}+\nabla
_{l}Y_{k}, $ $\overline{S}=P_{kl}+P_{lk}+\nabla _{k}X_{l}+\nabla _{l}X_{k},$ 
$\overline{K}=K_{kl}+K_{lk}.$ Symmetrizing indices in (\ref{II1}) and taking
into consideration equations (\ref{I1}), (\ref{III1}) and (\ref{LE5-1}) we
obtain a homogeneous system of linear equations in $\overline{X},$ $%
\overline{Y},$ $\overline{S},$ $\overline{K}:$%
\begin{equation}
\begin{bmatrix}
A & a_{2} & 0 & 0 \\ 
a_{2} & a_{1} & a_{2} & A \\ 
0 & 0 & a_{1} & a_{2} \\ 
0 & 0 & 2b & b_{2}%
\end{bmatrix}%
\begin{bmatrix}
\overline{X} \\ 
\overline{Y} \\ 
\overline{S} \\ 
\overline{K}%
\end{bmatrix}%
=%
\begin{bmatrix}
0 \\ 
0 \\ 
0 \\ 
0%
\end{bmatrix}%
,  \label{macierz1}
\end{equation}%
where $b=b_{1}-a_{1}^{\prime }.$ The system has a unique solution if and
only if 
\begin{equation*}
a(2ba_{2}-a_{1}b_{2})\neq 0,
\end{equation*}%
where $a=a_{1}A-a_{2}^{2}.$

Suppose $2ba_{2}-a_{1}b_{2}=0.$

If $a_{2}b_{2}\neq 0,$ then multiplying the third equation by $b_{2}$ and
the fourth one by $a_{2}$ we transform the whole system to%
\begin{equation*}
\begin{bmatrix}
A & a_{2} & 0 \\ 
a_{2}^{2} & a_{1}a_{2} & -a \\ 
0 & 0 & a_{1}%
\end{bmatrix}%
\begin{bmatrix}
\overline{X} \\ 
\overline{Y} \\ 
\overline{S}%
\end{bmatrix}%
=%
\begin{bmatrix}
0 \\ 
0 \\ 
-a_{2}\overline{K}%
\end{bmatrix}%
\end{equation*}%
with determinant equal to $a_{1}a_{2}a.$

Therefore, if $a_{1}\neq 0$ and $a_{2}b_{2}\neq 0,$ we get 
\begin{equation}
\overline{X}+\overline{S}=0,\quad \overline{Y}=\frac{A}{a_{2}}\overline{S}%
,\quad \overline{K}=-\frac{a_{1}}{a_{2}}\overline{S}.  \label{solution2}
\end{equation}%
On the other hand, if $a_{1}=0$ and $a_{2}b_{2}\neq 0,$ then $b=0$ and (\ref%
{macierz1}) yields%
\begin{equation*}
\begin{bmatrix}
A & a_{2} & 0 & 0 \\ 
a_{2} & 0 & a_{2} & A \\ 
0 & 0 & 0 & a_{2} \\ 
0 & 0 & 0 & b_{2}%
\end{bmatrix}%
\begin{bmatrix}
\overline{X} \\ 
\overline{Y} \\ 
\overline{S} \\ 
\overline{K}%
\end{bmatrix}%
=%
\begin{bmatrix}
0 \\ 
0 \\ 
0 \\ 
0%
\end{bmatrix}%
,
\end{equation*}%
whence%
\begin{equation*}
\overline{X}+\overline{S}=0,\quad A\overline{X}+a_{2}\overline{Y}=0,\quad 
\overline{K}=0.
\end{equation*}

Now suppose $a_{2}=0.$ Then by $2ba_{2}-a_{1}b_{2}=0$ we have either $%
a_{1}=0 $ or $b_{2}=0.$ But $a_{1}=a_{2}=0$ would give $a=0.$ On the other
hand $a_{2}=b_{2}=0$ reduce the system (\ref{macierz1}) to 
\begin{equation*}
\begin{bmatrix}
A & 0 & 0 & 0 \\ 
0 & a_{1} & 0 & A \\ 
0 & 0 & a_{1} & 0 \\ 
0 & 0 & 2b & 0%
\end{bmatrix}%
\begin{bmatrix}
\overline{X} \\ 
\overline{Y} \\ 
\overline{S} \\ 
\overline{K}%
\end{bmatrix}%
=%
\begin{bmatrix}
0 \\ 
0 \\ 
0 \\ 
0%
\end{bmatrix}%
.
\end{equation*}%
Since $a_{2}=0$ and $a\neq 0$ hold if $a_{1}A\neq 0,$ we obtain 
\begin{equation*}
\overline{X}=0,\quad \overline{S}=0,\quad A\overline{K}+a_{1}\overline{Y}=0.
\end{equation*}%
Finally, if $b_{2}=0$ but $a_{2}\neq 0,$ we have $b=b_{1}-a_{1}^{\prime }=0$
and from (\ref{macierz1}) we easily get (\ref{solution2}). Thus we have
proved

\begin{lemma}
Under assumption $a=a_{1}A-a_{2}^{2}\neq 0$ the system (\ref{macierz1}) has
the following solutions:

\begin{enumerate}
\item If $2ba_{2}-a_{1}b_{2}\neq 0,$ then $\overline{X}=\overline{Y}=%
\overline{S}=\overline{K}=0.$

\item If $2ba_{2}-a_{1}b_{2}=0$ and either $a_{1}a_{2}b_{2}\neq 0$ or $%
b_{2}=0$ and $a_{2}\neq 0$ and $b=0,$ then $\overline{X}+\overline{S}=0,\ 
\overline{Y}=\frac{A}{a_{2}}\overline{S},\ \overline{K}=-\frac{a_{1}}{a_{2}}%
\overline{S}.$

\item If $a_{1}=b=0,$ then $\overline{X}+\overline{S}=0,\ A\overline{X}+a_{2}%
\overline{Y}=0,\ \overline{K}=0.$

\item If $a_{2}=b_{2}=0,$ then $\overline{X}=0,\ \overline{S}=0,\ A\overline{%
K}+a_{1}\overline{Y}=0.$
\end{enumerate}

Conversely, if $a\neq 0,$ then the above four cases give the only possible
solutions to (\ref{macierz1}).
\end{lemma}

Combining the above lemma with (\ref{I1}), (\ref{II1}), (\ref{III1}) and (%
\ref{LE5-1}) we obtain the following

\begin{theorem}
\label{Splitting theorem}Let $(TM,$ $G)$ be a tangent bundle of a Riemannian
manifold $(M,g),$ dim$M>2,$ with g natural metric $G$ such that $%
a=a_{1}A-a_{2}^{2}\neq 0$ on $M\times \{0\}.$ Let $Z$ be a Killing vector
field on $TM$ with its Taylor series expansion around a point $(x,0)\in TM$
given by (\ref{Taylor1}). Then for each such a point there exists a
neighbourhood $U\subset M,$ $x\in U$, that one of the following cases occurs:

\begin{enumerate}
\item $2ba_{2}-a_{1}b_{2}\neq 0.$ Then 
\begin{eqnarray}
\nabla _{k}X_{l}+\nabla _{l}X_{k} &=&0,\quad \nabla _{k}Y_{l}+\nabla
_{l}Y_{k}=0,  \label{14} \\
P_{kl}+P_{lk} &=&0,\quad K_{kl}+K_{lk}=0.  \label{15}
\end{eqnarray}

\item $2ba_{2}-a_{1}b_{2}=0$ and either $a_{1}a_{2}b_{2}\neq 0$ or $%
a_{2}\neq 0$ and $b_{2}=0.$ Then%
\begin{eqnarray}
P_{kl}+P_{lk}+2\left( \nabla _{k}X_{l}+\nabla _{l}X_{k}\right) &=&0,
\label{16} \\
a_{2}\left( \nabla _{k}Y_{l}+\nabla _{l}Y_{k}\right) +A\left( \nabla
_{k}X_{l}+\nabla _{l}X_{k}\right) &=&0,  \label{17} \\
a_{2}\left( K_{kl}+K_{lk}\right) -a_{1}\left( \nabla _{k}X_{l}+\nabla
_{l}X_{k}\right) &=&0.  \label{18}
\end{eqnarray}

\item $a_{2}b_{2}\neq 0$ and $a_{1}=b=0.$ Then%
\begin{eqnarray}
P_{kl}+P_{lk}+2\left( \nabla _{k}X_{l}+\nabla _{l}X_{k}\right) &=&0,
\label{19} \\
a_{2}\left( \nabla _{k}Y_{l}+\nabla _{l}Y_{k}\right) +A\left( \nabla
_{k}X_{l}+\nabla _{l}X_{k}\right) &=&0,  \label{20} \\
K_{kl}+K_{lk} &=&0.  \label{21}
\end{eqnarray}

\item $a_{2}=b_{2}=0.$ Then%
\begin{equation}
\nabla _{k}X_{l}+\nabla _{l}X_{k}=0,\quad P_{kl}+P_{lk}=0,\quad
AK_{lk}+a_{1}\nabla _{l}Y_{k}=0.  \label{22}
\end{equation}
\end{enumerate}
\end{theorem}

In the above theorem we have put $a_{j}=a_{j}(r^{2})_{\mid (x,0)\in TM},$ $%
b_{j}=b_{j}(r^{2})_{\mid (x,0)\in TM},$ $a_{j}^{\prime }=a_{j}^{\prime
}(r^{2})_{\mid (x,0)\in TM},$ $A=a_{1}+a_{3}.$

It is clear that the above results together with Proposition \ref{Lift prop
2} yield Theorem \ref{mine}.

\subsection{Case 1}

In this section we study relations between $Y$ components of the Killing
vector field on $TM$ and the base manifold $M$ (Theorems \ref{LEC1b}, \ref%
{LEC1c}). Various conditions for $Y$ to be non-zero and relations between $%
X, $ $Y,$ $P$, $K$ are proved. Moreover, Theorem \ref{Structure 1}
establishes isomorphism between algebras of Killing vector fields on $M$ and 
$TM$ for large subclass of $g-$ metrics.

\begin{lemma}
\label{LEC1a}Under hypothesis (\ref{H}) suppose $\dim M>2,$ $a\neq 0$ and $%
2(b_{1}-a_{1}^{\prime })a_{2}-a_{1}b_{2}\neq 0$ at a point $(x,0)\in TM.$
Then 
\begin{equation}
(B+A^{\prime })Y_{k}=0,  \label{LEC1a1}
\end{equation}%
\begin{equation}
2a\nabla _{l}K_{km}=\left[ 2a_{1}A^{\prime }+a_{2}(b_{2}-2a_{2}^{\prime })%
\right] (g_{lm}Y_{k}-g_{lk}Y_{m}),  \label{LEC1a2}
\end{equation}%
\begin{equation}
2a\nabla _{l}P_{km}=-\left[ 2a_{2}A^{\prime }+A(b_{2}-2a_{2}^{\prime })%
\right] (g_{lm}Y_{k}-g_{lk}Y_{m}),  \label{LEC1a3}
\end{equation}%
\begin{equation}
a_{1}\nabla _{m}\nabla _{l}Y_{k}=A^{\prime }(g_{ml}Y_{k}-g_{mk}Y_{l}),
\label{LEC1a4}
\end{equation}%
\begin{equation}
a_{1}Y^{r}R_{rklm}=A^{\prime }(g_{km}Y_{l}-g_{kl}Y_{m})  \label{LAC1a5}
\end{equation}%
hold at the point.
\end{lemma}

\begin{proof}
First suppose $a_{1}\neq 0.$ Symmetrizing (\ref{L6a}) in $(k,m),$ making use
of the skew-symmetricity of $K,$ then alternating in $(k,l)$ and applying
the first Bianchi identity, we get%
\begin{equation}
3a_{1}Y^{r}R_{rmkl}+(B-2A^{\prime })(g_{lm}Y_{k}-g_{km}Y_{l})=0.
\label{LEC1a6}
\end{equation}%
Applying the last identity to (\ref{L6a}) we find 
\begin{multline*}
6a\nabla _{l}K_{km}+2a_{1}(B+A^{\prime })g_{km}Y_{l}+3\left[ 2a_{1}A^{\prime
}+a_{2}(b_{2}-2a_{2}^{\prime })\right] g_{lk}Y_{m}+ \\
\left[ 2a_{1}(2B-A^{\prime })-3a_{2}(b_{2}-2a_{2}^{\prime })\right]
g_{lm}Y_{k}=0,
\end{multline*}%
whence, symmetrizing in $(k,m),$ we obtain (\ref{LEC1a1}) and, consequently,
(\ref{LEC1a2}).

Suppose now $a_{1}=0.$ Substituting in (\ref{L6a}) we easily state that (\ref%
{LEC1a2}) remains true. On the other hand, substituting $a_{1}=0$ into (\ref%
{L6b}) and symmetrizing in $(k,m)$ we get%
\begin{equation*}
2a_{2}Bg_{km}Y_{l}+a_{2}(B+2A^{\prime })(g_{lm}Y_{k}+g_{lk}Y_{m})=0,
\end{equation*}%
whence, by contractions with $g^{km}$ and $g^{lm},$ we obtain 
\begin{equation}
BY_{l}=0\text{ and }A^{\prime }Y_{l}=0  \label{LEC1a5}
\end{equation}%
respectively since $a_{2}\neq 0$ must hold. Thus (\ref{LEC1a1}) holds good.

Since $X$ is a Killing vector field, (\ref{L6b}), (\ref{Conv5}), (\ref%
{LEC1a1}) and (\ref{LEC1a6}) in the case $a_{1}\neq 0$ and (\ref{L6b}) and (%
\ref{LEC1a5}) as well in the case $a_{1}=0$ yield (\ref{LEC1a3}).

Differentiating covariantly (\ref{II1}), using just obtained identities, we
get (\ref{LEC1a4}). Finally, alternating (\ref{LEC1a4}) in $(l,m),$ by the
use of the Ricci identity (\ref{Conv4}), we obtain (\ref{LAC1a5}). This
completes the proof.
\end{proof}

From (\ref{LAC1a5}) and Theorem \ref{Apen1} by Grycak we infer

\begin{theorem}
\label{LEC1b}Under hypothesis (\ref{H}) suppose $\dim M>2,$ $a\neq 0$ and $%
2(b_{1}-a_{1}^{\prime })a_{2}-a_{1}b_{2}\neq 0$ on the set $M\times
\{0\}\subset TM.$ If the vector field $\frac{A^{\prime }}{a_{1}}%
Y^{a}\partial _{a}$ does not vanish on a dense subset of $M$ and $M$ is
semisymmetric, i.e. $R\cdot R=0,$ (resp. the Ricci tensor $S$ is
semisymmetric, i.e. $R\cdot S=0$), then $M$ is a space of constant
curvature, (resp. $M$ is an Einstein manifold).
\end{theorem}

\begin{theorem}
\label{LEC1c}Under hypothesis (\ref{H}) suppose $\dim M>2,$ $a\neq 0$ and $%
2(b_{1}-a_{1}^{\prime })a_{2}-a_{1}b_{2}\neq 0$ at a point $(x,0)\in TM.$
Then the $Y$ component of the Killing vector field on $TM$ satisfies%
\begin{equation}
S_{1}Y\left[ a_{1}R+\frac{B}{2}g\wedge g\right] =0
\label{Th Eq II4 and III3}
\end{equation}%
on $M.$
\end{theorem}

\begin{proof}
Suppose $a_{1}\neq 0.$ By (\ref{14}) and (\ref{15}) we have $\overline{S}%
_{ab}=0.$ Applying this and (\ref{15}), (\ref{LEC1a1}), (\ref{LEC1a2}) and (%
\ref{LEC1a5}) to Lemma \ref{LE6''}, after long computations we obtain%
\begin{multline}
S_{1}\left[ 3(R_{blck}Y_{a}+R_{alck}Y_{b})+\left( R_{akbl}+R_{albk}\right)
Y_{c}-\left( R_{acbl}+R_{albc}\right) Y_{k}\right] +  \label{Case 1-39} \\
S_{2}g_{ab}\left( g_{kl}Y_{c}-g_{cl}Y_{k}\right) +S_{3}\left[ \left(
g_{al}g_{bk}+g_{ak}g_{bl}\right) Y_{c}-\left(
g_{al}g_{bc}+g_{ac}g_{bl}\right) Y_{k}\right] + \\
S_{4}\left[ (g_{bk}g_{cl}-g_{bc}g_{kl})Y_{a}+(g_{ak}g_{cl}-g_{ac}g_{kl})Y_{b}%
\right] =0,
\end{multline}%
where 
\begin{equation*}
S_{1}=a_{1}\left[ 2a_{2}a_{1}^{\prime }-a_{1}\left( b_{2}+2a_{2}^{\prime
}\right) \right] ,
\end{equation*}%
\begin{multline*}
S_{2}=-2\left[ b_{2}\left( -Ab_{1}+3a_{2}b_{2}+5a_{1}A^{\prime
}-Aa_{1}^{\prime }-4a_{2}a_{2}^{\prime }\right) +2b_{1}(Aa_{2}^{\prime
}-a_{2}A^{\prime })\right. + \\
\left. 2(a_{1}A^{\prime }+Aa_{1}^{\prime }-2a_{2}a_{2^{\prime
}})a_{2}^{\prime }\right] = \\
-2\left[ b_{2}\left( -Ab_{1}+3\left( a_{2}b_{2}+a_{1}A^{\prime
}-Aa_{1}^{\prime }\right) +2a^{\prime }\right) +2b_{1}(Aa_{2}^{\prime
}-a_{2}A^{\prime })+2a^{\prime }a_{2}^{\prime }\right. ],
\end{multline*}%
\begin{multline*}
S_{3}=-3a_{1}b_{2}A^{\prime }-2Ab_{2}a_{1}^{\prime }+2a_{2}A^{\prime
}a_{1}^{\prime }+4a_{2}a_{2}^{\prime }b_{2}-2a_{1}A^{\prime }a_{2}^{\prime
}+4ab_{2}^{\prime }= \\
2A^{\prime }(a_{2}a_{1}^{\prime }-a_{1}a_{2}^{\prime })-b_{2}(2a^{\prime
}+a_{1}A^{\prime })+4ab_{2}^{\prime },
\end{multline*}%
\begin{multline*}
S_{4}=b_{2}\left( -2Ab_{1}+6a_{2}b_{2}+7a_{1}A^{\prime }-4Aa_{1}^{\prime
}-4a_{2}a_{2}^{\prime }\right) -4a_{2}b_{1}A^{\prime }+2a_{2}A^{\prime
}a_{1}^{\prime }+ \\
a_{2}^{\prime }\left( 4Ab_{1}+2a_{1}A^{\prime }+4Aa_{1}^{\prime
}-8a_{2}a_{2}^{\prime }\right) +4ab_{2}^{\prime }
\end{multline*}%
and 
\begin{equation*}
S_{2}-S_{3}+S_{4}=0
\end{equation*}%
identically.

Symmetrizing (\ref{Case 1-39}) in $(a,b,l)$ we get%
\begin{equation*}
(S_{2}+2S_{3})\left[ \left( g_{al}g_{bk}+g_{ak}g_{lb}+g_{ab}g_{kl}\right)
Y_{c}-\left( g_{al}g_{bc}+g_{ac}g_{lb}+g_{ab}g_{cl}\right) Y_{k}\right] =0,
\end{equation*}%
whence, by contraction with $g^{al}g^{bk},$ we find $%
(n-1)(n+2)(S_{2}+2S_{3})Y_{c}=0$. Therefore, symmetrizing (\ref{Case 1-39})
in $(a,b,c)$ and using the last result, we obtain%
\begin{equation*}
Y_{a}T_{bckl}+Y_{b}T_{cakl}+Y_{c}T_{abkl}=0,
\end{equation*}%
where%
\begin{multline*}
T_{bckl}=T_{cbkl}=T_{klbc}= \\
2S_{1}(R_{bkcl}+R_{blck})-(S_{3}+S_{4})\left[ g_{bc}g_{kl}-\frac{1}{2}%
(g_{bl}g_{ck}+g_{bk}g_{cl})\right] .
\end{multline*}%
Hence, by the use of the Walker's Lemma \ref{Apen2}, we get%
\begin{equation}
Y_{a}T_{bckl}=0.  \label{Case 1-46}
\end{equation}

Alternating (\ref{Case 1-46}) in $(l,c)$ and applying the Bianchi identity
we obtain%
\begin{equation*}
Y_{a}\left[ 4S_{1}R_{bkcl}+\left( S_{3}+S_{4}\right) \left(
g_{bl}g_{kc}-g_{bc}g_{kl}\right) \right] =0.
\end{equation*}%
Transvecting the last equation with $Y^{b},$ by the use of (\ref{LEC1a5}),
we easily get 
\begin{equation*}
\left[ 4BS_{1}+a_{1}(S_{3}+S_{4})\right] Y_{a}=0,
\end{equation*}%
whence (\ref{Th Eq II4 and III3}) results.

On the other hand, from the proof of Lemma \ref{LEC1a} it follows that $%
a_{1}(0)=0$ implies $B(0)Y_{a}=0.$ Thus, by continuity, (\ref{Th Eq II4 and
III3}) holds good on $M.$
\end{proof}

\begin{corollary}
Under assumptions of the above theorem we have on $M:$%
\begin{eqnarray*}
\left( S_{2}+2S_{3}\right) Y &=&0\text{ if }a_{1}\neq 0, \\
\left[ 4BS_{1}+a_{1}(S_{3}+S_{4})\right] Y &=&0.
\end{eqnarray*}%
Notice that multiplying the first equation by $a_{1}$ and adding to the
second one we obtain%
\begin{equation*}
a_{1}\left( b_{2}a^{\prime }-2ab_{2}^{\prime }\right) Y=0.
\end{equation*}
\end{corollary}

\begin{lemma}
Under hypothesis (\ref{H}) suppose $\dim M>2,$ $a\neq 0$ and $%
2(b_{1}-a_{1}^{\prime })a_{2}-a_{1}b_{2}\neq 0$ at a point $(x,0)\in TM.$

If $a_{1}a_{2}\neq 0,$ then%
\begin{multline}
A_{km}=\left[ 2a_{2}\left( b_{1}-a_{1}^{\prime }\right) -\frac{3}{2}%
a_{1}\left( b_{2}-2a_{2}^{\prime }\right) \right] Y_{k,m}+  \label{LEC1b1} \\
\left( 3a_{1}B-a_{2}b_{2}\right) P_{km}+3a_{2}BK_{km}=0.
\end{multline}

If $a_{2}=0$ and $a_{1}b_{2}\neq 0$ then%
\begin{equation}
A_{km}=-\frac{1}{2}a_{1}\left( b_{2}-2a_{2}^{\prime }\right)
Y_{k,m}+a_{1}BP_{km}=0.  \label{LEC1b2}
\end{equation}

If $a_{1}=0$ and $(b_{1}-a_{1}^{\prime })a_{2}\neq 0$ then%
\begin{eqnarray}
(n+1)BK_{kn}-b_{2}P_{kn}+2(b_{1}-a_{1}^{\prime })Y_{k,n} &=&0,
\label{LEC1b3} \\
3BK_{ln}-(n-1)b_{2}P_{ln}+2(n-1)(b_{1}-a_{1}^{\prime })Y_{l,n} &=&0.
\label{LEC1b4}
\end{eqnarray}
\end{lemma}

\begin{proof}
If $a_{1}a_{2}\neq 0,$ we apply (\ref{14}) and (\ref{15}) to (\ref{LE8a}) to
obtain (\ref{LEC1b1}).

If $a_{2}=0$ but $a_{1}\neq 0,$ then also there must be $b_{2}\neq 0.$
Substituting $a_{2}=0$ into (\ref{LE8a}) and applying (\ref{14}) and (\ref%
{15}) we get (\ref{LEC1b2}).

Finally, the last two identities one obtains substituting $a_{1}=0$ into (%
\ref{like 78}), contracting with $g^{km}$ and $g^{lm}$ and making use of (%
\ref{14}) and (\ref{15}).
\end{proof}

Taking into account (\ref{LEC1b3}) and (\ref{LEC1b4}) together with the
equation (\ref{II1}) which, in virtue of (\ref{14}), writes%
\begin{equation*}
AK_{km}+a_{2}P_{km}-a_{1}Y_{k,m}=0
\end{equation*}%
we find that $B\neq 0$ implies $P=K=\nabla Y=0$ on $M.$ We conclude with the
following

\begin{theorem}
\label{Structure 1}Let $TM,$ $dimTM>4,$ be endowed with a $g$-natural metric 
$G,$ such that $a_{1}=0,$ $(b_{1}-a_{1}^{\prime })a_{2}\neq 0$ and $B\neq 0$
on $M\times \{0\}\subset TM$. Let $V$ be an open subset of $TM$ such that $%
M\times \{0\}\subset V.$ If $V$admits a Killing vector field, then it is a
complete lift of a Killing vector field on $M.$ Consequently, Lie algebras
of Killing vector fields on $M$ and $V\subset TM$ are isomorphic.
\end{theorem}

Besides, for $B=0,$ we have

\begin{theorem}
Let $TM,$ $dimTM>4,$ be endowed with a $g$-natural metric $G,$ such that $%
a_{1}=0,$ $(b_{1}-a_{1}^{\prime })a_{2}\neq 0$ and $B=0$ on $M\times
\{0\}\subset TM$. Then%
\begin{eqnarray*}
a_{2}P+AK &=&0, \\
b_{2}P-2(b_{1}-a_{1}^{\prime })\nabla Y &=&0
\end{eqnarray*}%
hold on $M\times \{0\}\subset TM$.
\end{theorem}

Hence, for $B=0,$ $A\neq 0$ and $b_{2}\neq 0,$ a theorem similar to the
former one can be deduced.

The next theorem gives further restrictions on the vector $Y$ to be non-zero.

\begin{theorem}
\label{Th7C1}Under hypothesis (\ref{H}) suppose $\dim M>2,$ $a\neq 0$ and $%
2(b_{1}-a_{1}^{\prime })a_{2}-a_{1}b_{2}\neq 0$ at a point $(x,0)\in TM.$ If 
$a_{1}\neq 0,$ then the $Y$ component of the Killing vector field on $TM$
satisfies%
\begin{equation*}
Q_{2}Y=\left\{ a_{1}b_{2}\left[ A(b_{2}-2a_{2}^{\prime })-2a_{2}B\right]
-4aB(b_{1}-a_{1}^{\prime }))\right\} Y=0,
\end{equation*}%
\begin{equation*}
B^{\prime }Y=0,
\end{equation*}%
\begin{equation*}
B\left[ a_{1}a_{2}\left( b_{2}+2a_{2}^{\prime }\right) -2Aa_{1}a_{1}^{\prime
}+aa_{1}^{\prime }\right] Y=0.
\end{equation*}
\end{theorem}

\begin{proof}
We apply Lemma \ref{LE6'}. By the use of (\ref{14}), (\ref{15}), (\ref%
{LEC1a1}) - (\ref{LEC1a4}) and (\ref{L5a2}) the components of the tensors $K$
and $L$ defined by (\ref{LE6'-2}) and (\ref{LE6'-3}) can be written as%
\begin{equation*}
K_{kal}=\frac{\left[ aB(b_{1}-a_{1}^{\prime
})+2a_{1}a_{2}Bb_{2}-Aa_{1}b_{2}(b_{2}-2a_{2}^{\prime })\right] }{aa_{1}}%
(2g_{kl}Y_{a}-g_{ka}Y_{l}-g_{la}Y_{k}),
\end{equation*}%
\begin{multline*}
L_{abl}=3BT_{lab}+3B^{\prime }(g_{bl}Y_{a}+g_{al}Y_{b})= \\
-\frac{3B\left[ A(b_{1}-a_{1}^{\prime })-a_{2}b_{2}\right] }{a}g_{ab}Y_{l}+
\\
\frac{3\left[ B(a_{2}b_{2}-2Aa_{1}^{\prime }+2a_{2}a_{2}^{\prime
})+2aB^{\prime }\right] }{2a}\left( g_{al}Y_{b}+g_{bl}Y_{a}\right) .
\end{multline*}%
Substituting into (\ref{LE6'-1}) and applying (\ref{L5a1}), (\ref{L5a2'})
and (\ref{LEC1a5}) we get%
\begin{multline}
Q_{1}\left[ \left( R_{bkcl}+R_{blck}\right) Y_{a}+\left(
R_{ckal}+R_{clak}\right) Y_{b}+\left( R_{akbl}+R_{albk}\right) Y_{c}\right] +
\label{Th7C1-6} \\
Q_{2}\left[ \left( g_{al}g_{bc}+g_{bl}g_{ca}+g_{cl}g_{ab}\right)
Y_{k}+\left( g_{ak}g_{bc}+g_{bk}g_{ca}+g_{ck}g_{ab}\right) Y_{l}\right] + \\
Q_{3}g_{kl}\left( g_{bc}Y_{a}+g_{ca}Y_{b}+g_{ab}Y_{c}\right) + \\
Q_{4}\left[ \left( g_{bl}g_{kc}+g_{bk}g_{lc}\right) Y_{a}+\left(
g_{cl}g_{ka}+g_{ck}g_{la}\right) Y_{b}+\left(
g_{al}g_{kb}+g_{ak}g_{lb}\right) Y_{c}\right] =0,
\end{multline}%
where%
\begin{equation*}
Q_{1}=-\frac{3a_{2}\left( a_{1}b_{2}-2a_{2}a_{1}^{\prime
}+2a_{1}a_{2}^{\prime }\right) }{a},
\end{equation*}%
\begin{equation*}
Q_{2}=\frac{\left[ a_{1}b_{2}(A(b_{1}-a_{1}^{\prime
})-2Ba_{2})-4aB(b_{1}-a_{1}^{\prime }))\right] }{aa_{1}},
\end{equation*}%
\begin{equation*}
Q_{3}=2\frac{4aB(b_{1}-a_{1}^{\prime })-a_{1}\left[ A(b_{2}-2a_{2}^{\prime
})+B(a_{2}b_{2}-6Aa_{1}^{\prime }+6a_{2}a_{2}^{\prime })\right] }{aa_{1}},
\end{equation*}%
\begin{equation*}
Q_{4}=\frac{3\left[ B\left( a_{2}b_{2}-2Aa_{1}^{\prime }+2a_{2}a_{2}^{\prime
}\right) +2aB^{\prime }\right] }{a}.
\end{equation*}%
Contracting (\ref{Th7C1-6}) with $g^{ab}$, by the use of (\ref{LEC1a5}), we
get%
\begin{multline}
g_{kl}\left( -\frac{4BQ_{1}}{a_{1}}+(n+2)Q_{3}+2Q_{4}\right)
Y_{c}-2Q_{1}R_{kl}Y_{c}+  \label{Th7C1-7} \\
\left( \frac{2BQ_{1}}{a_{1}}+(n+2)Q_{2}+2Q_{4}\right)
(g_{cl}Y_{k}+g_{kc}Y_{l})=0.
\end{multline}%
Symmetrizing in $(c,k,l)$ we obtain%
\begin{equation*}
T_{kl}Y_{c}+T_{lc}Y_{k}+T_{ck}Y_{l}=0,
\end{equation*}%
where 
\begin{equation}
T_{kl}=T_{lk}=g_{kl}\left[ \left( n+2\right) \left( 2Q_{2}+Q_{3}\right)
+6Q_{4}\right] -2Q_{1}R_{kl}.  \label{Th7C1-8}
\end{equation}%
Then the Walker lemma yields $T_{kl}=0$ or $Y_{c}=0.$ Subtracting (\ref%
{Th7C1-8}) from (\ref{Th7C1-7}) and contracting with $g^{kl}$ we get%
\begin{equation}
\left[ a_{1}\left( (n+2)Q_{2}+2Q_{4}\right) +2BQ_{1}\right] Y_{c}=0.
\label{Th7C1-9}
\end{equation}

In the same way, by contraction of (\ref{Th7C1-6}) with $g^{cl},$ we find%
\begin{equation}
\left\{ g_{bk}\left[ (n+5)Q_{2}+3Q_{3}+2(n+2)Q_{4}\right] +2Q_{1}R_{bk}%
\right\} Y_{c}=0  \label{Th7C1-10}
\end{equation}%
and%
\begin{equation}
\left[ a_{1}\left( \left( n+3\right) Q_{2}+Q_{3}\right) -2BQ_{1}\right]
Y_{k}=0.  \label{Th7C1-11}
\end{equation}

At last, by contraction of (\ref{Th7C1-6}) with $g^{kl},$ we obtain%
\begin{equation}
\left[ g_{bc}\left( 2Q_{2}+nQ_{3}+2Q_{4}\right) -2Q_{1}R_{bc}\right] Y_{a}=0.
\label{Th7C1-12}
\end{equation}

Eliminating the Ricci tensor between (\ref{Th7C1-8}), (\ref{Th7C1-12}) and (%
\ref{Th7C1-10}) we find 
\begin{equation*}
\left[ 3(n+3)Q_{2}+(n+5)(Q_{3}+2Q_{4})\right] Y_{c}=0,
\end{equation*}%
\begin{equation*}
\left[ (n+1)Q_{2}+2Q_{3}+2Q_{4}\right] Y_{c}=0.
\end{equation*}%
The system consisting of (\ref{Th7C1-9}), (\ref{Th7C1-11}) and the above two
equations is undetermined and equivalent to $Y=0$ or $Q_{2}=0$ and $%
2BQ_{1}+a_{1}Q_{3}=0$ and $Q_{3}+2Q_{4}=0.$ Hence $2Q_{2}+Q_{3}+2Q_{4}$
yields the second identity, while $a_{1}(Q_{3}+2Q_{4})-(2BQ_{1}+a_{1}Q_{3})$
gives the third one.
\end{proof}

\begin{remark}
From (\ref{Th7C1-6}) one can deduce the identity%
\begin{equation*}
Q_{1}Y\left[ a_{1}R+\frac{B}{2}g\wedge g\right] =0.
\end{equation*}
\end{remark}

\subsection{Case 2}

The next theorem partially improves the result of Tanno (\cite{Tanno 1})
concerned with Killing vector field on $(TM,g^{C}),$ where the complete lift 
$g^{C}$ of $g$ is \ a $g-$ natural metric with $a_{2}=1,$ all others being
zero. (In Tanno's paper the Killing vector on $(TM,g^{C})$ is of the form $%
\iota C^{[X]}+X^{C}+Y^{v}+(u^{r}P_{r}^{t})\partial _{t}^{h},$ where $Y$ and $%
P$ satisfy some additional conditions). Furthermore, we prove in the section
some sufficient conditions for $X$ and $Y$ to be either infinitesimal affine
transformation or infinitesimal isometry.

\begin{theorem}
Let $X$ be an infinitesimal affine vector field on some open $U\subset M.$ If%
\begin{equation*}
a_{2}=const\neq 0,\text{ }b_{3}=const,\text{ }all\ others\ equal\text{ }0
\end{equation*}
on $\pi ^{-1}(U)\subset TM,$ then $\iota C^{[X]}+X^{C}$ is a Killing vector
field on $\pi ^{-1}(U).$
\end{theorem}

\begin{proof}
It follows from the results of subsection \ref{Subsection543}.
\end{proof}

\begin{lemma}
Under hypothesis (\ref{H}) suppose $\dim M>2,$ $a\neq 0$ and $%
2(b_{1}-a_{1}^{\prime })a_{2}-a_{1}b_{2}=0$ at a point $(x,0)\in TM.$
Moreover, let either $a_{1}a_{2}b_{2}\neq 0$ or $a_{2}\neq 0,$ $b_{2}=0,$ $%
b_{1}-a_{1}^{\prime }=0.$ Then 
\begin{equation*}
(a_{1}B-2a_{2}b_{2}-3a_{1}A^{\prime }+4a_{2}a_{2}^{\prime })\left[ \left(
L_{X}g\right) -\frac{1}{n}Tr(L_{X}g)g\right] =0,
\end{equation*}%
\begin{equation*}
a_{2}(b_{1}-a_{1}^{\prime })\left[ \left( L_{Y}g\right) -\frac{1}{n}%
Tr(L_{Y}g)g\right] =0,
\end{equation*}%
\begin{equation*}
a_{1}\left[ a_{2}^{\prime }\left( L_{Y}g\right) +A^{\prime }\left(
L_{X}g\right) \right] =0,
\end{equation*}%
\begin{equation*}
\left[ a_{1}(B-3A^{\prime })+A(b_{1}-a_{1}^{\prime
})-2a_{2}(b_{2}-2a_{2}^{\prime })\right] \left( L_{X}g\right) =0.
\end{equation*}
\end{lemma}

\begin{proof}
First consider the case $a_{1}a_{2}b_{2}\neq 0.$ By the use of (\ref{16}) - (%
\ref{18}) and the equality $a_{1}b_{2}=2a_{2}(b_{1}-a_{1}^{\prime })$ Lemma %
\ref{LE8} yields%
\begin{equation*}
\mathbf{F}=2(a_{1}B-2a_{2}b_{2}-3a_{1}A^{\prime }+4a_{2}a_{2}^{\prime
})(L_{X}g),
\end{equation*}%
\begin{equation*}
\mathbf{B}=-2a_{2}(b_{1}-a_{1}^{\prime })(L_{Y}g),
\end{equation*}%
whence, by Lemma \ref{Apen3}, the first two equalities result. Moreover, by
Lemma \ref{LE8} we have%
\begin{equation}
\mathbf{F}+\mathbf{B}=-2a_{2}(b_{1}-a_{1}^{\prime
})(L_{Y}g)+2(a_{1}B-2a_{2}b_{2}-3a_{1}A^{\prime }+4a_{2}a_{2}^{\prime
})(L_{X}g)=0,  \label{Case 2-30}
\end{equation}%
and%
\begin{multline*}
\mathbf{A}%
_{km}=3a_{2}BK_{km}+(3a_{1}B-a_{2}b_{2})P_{km}+(a_{1}B-2a_{2}a_{2}^{\prime
})(L_{X}g)_{km}+ \\
\left[ a_{2}\left( b_{1}-a_{1}^{\prime }\right) -3a_{1}a_{2}^{\prime }\right]
\nabla _{k}Y_{m}=0.
\end{multline*}%
Symmetrizing in $(k,m)$ and transforming the obtained equation in the same
manner as before we find%
\begin{equation}
\left[ a_{2}\left( b_{1}-a_{1}^{\prime }\right) -3a_{1}a_{2}^{\prime }\right]
\left( L_{Y}g\right) -(a_{1}B-2a_{2}b_{2}+4a_{2}a_{2}^{\prime })(L_{X}g)=0.
\label{Case2-32}
\end{equation}

Now from (\ref{Case 2-30}) and (\ref{Case2-32}) we easily deduce the third
equality. Finally, the last one is obtained by applying (\ref{17}) to (\ref%
{Case 2-30}).

A proof of the second case can be obtained in the same way. The statements
differ only in that $b_{2}=0.$
\end{proof}

\begin{corollary}
If $a_{1}(a_{2}A^{\prime }-a_{2}^{\prime }A)\neq 0,$ then $L_{X}g=0.$
\end{corollary}

\subsection{Case 3}

The main result of the section establishes isomorphism between algebras of
Killing vector fields on $M$ and $TM$ for large subclass of $g-$ metrics
(Theorem \ref{Structure 2}). Furthermore, conditions for $Y$ to be non-zero
are proved.

\begin{lemma}
\label{Case3 Lemma 2}Under hypothesis (\ref{H}) suppose that $dimM>2$ and
the following conditions on $a_{j},$ $b_{j}$ at a point $(x,0)\in M$ are
satisfied: $a_{1}=0,$ $b_{1}=a_{1}^{\prime },$ $a_{2}\neq 0,$ $b_{2}\neq 0.$
Then the relations%
\begin{equation}
\left( b_{2}-2a_{2}^{\prime }\right) L_{X}g=0,\quad \left(
b_{2}-2a_{2}^{\prime }\right) Tr\left( \nabla X\right) =0,\quad \left(
b_{2}-2a_{2}^{\prime }\right) TrP=0,  \label{Case3-1-1}
\end{equation}%
\begin{equation}
BK=0,\quad L_{X}g+P=0  \label{Case3-1-2}
\end{equation}%
hold. Moreover $P$ is symmetric. Finally $a_{3}K=0.$
\end{lemma}

\begin{proof}
Substituting $a_{1}=0$ and $a_{1}^{\prime }=b_{1}$ into (\ref{like 78}),
then applying (\ref{21}) and (\ref{19}) we find%
\begin{multline}
b_{2}\left[ 2g_{kn}\left( \left( L_{X}g\right) _{lm}+P_{lm}\right)
+g_{mn}\left( \left( L_{X}g\right) _{kl}+P_{kl}\right) -g_{km}\left( \left(
L_{X}g\right) _{ln}+P_{ln}\right) \right] +  \label{Case3-1-3} \\
g_{ln}\left[ -3BK_{km}+\left( b_{2}-2a_{2}^{\prime }\right) \left(
L_{X}g\right) _{km}\right] + \\
g_{kl}\left[ 3BK_{mn}+\left( b_{2}-2a_{2}^{\prime }\right) \left(
L_{X}g\right) _{mn}\right] -2(b_{2}-2a_{2}^{\prime })g_{lm}\left(
L_{X}g\right) _{kn}=0.
\end{multline}%
From (\ref{19}) it follows that $P_{a}^{a}+2X_{,a}^{a}=0.$ Thus contracting (%
\ref{Case3-1-3}) with $g^{lm}$ and then with $g^{kn}$ we get (\ref{Case3-1-1}%
) in turn. Consequently, contracting (\ref{Case3-1-3}) with $g^{kn},$ by the
use of (\ref{19}), (\ref{21}) and (\ref{Case3-1-1}), we obtain 
\begin{equation*}
-3BK_{lm}+(n-1)b_{2}\left[ P_{lm}+(L_{X}g)_{lm}\right] =0.
\end{equation*}%
In a similar way, contracting (\ref{Case3-1-1}) with $g^{kl},$ we find%
\begin{equation*}
-(n+1)BK_{mn}+b_{2}\left[ P_{mn}+(L_{X}g)_{mn}\right] =0.
\end{equation*}%
The last two equations yield (\ref{Case3-1-2}). The final statement is a
consequence of (\ref{Case3-1-2}), (\ref{II1}) and $a_{1}=0.$
\end{proof}

\begin{lemma}
\label{Case3 Lemma 4}Under assumptions of Lemma \ref{Case3 Lemma 2} relations%
\begin{equation*}
\left[ \left( b_{2}-2a_{2}^{\prime }\right)
(2Ab_{1}-3a_{2}b_{2}-2a_{2}a_{2}^{\prime })-2a_{2}Bb_{1}\right] Y=0,
\end{equation*}%
\begin{equation*}
\left[ a_{2}Bb_{1}+Ab_{1}b_{2}-2a_{2}\left( b_{2}a_{2}^{\prime
}-a_{2}b_{2}^{\prime }\right) \right] Y=0,
\end{equation*}%
\begin{equation*}
(b_{1}b_{2}-a_{2}b_{1}^{\prime })Y=0
\end{equation*}%
hold on $M\times \{0\}.$
\end{lemma}

\begin{proof}
We apply Lemma \ref{LE6''}. Substituting $a_{1}^{\prime }=b_{1},$ $a_{1}=0,$
contracting with $g^{ab}g^{cl}$ and applying (\ref{21}) we get%
\begin{equation*}
-2b_{2}(n+2)K_{\ k,r}^{r}+2BE_{\ kr}^{r}-2BE_{k\ r}^{\
r}+(n-1)(b_{2}-2a_{2}^{\prime })M_{\ kr}^{r}=0,
\end{equation*}%
whence, by the use of Lemma \ref{Lemmas L5} we obtain the first equality.
Similarly, contracting with $g^{al}g^{bc}$ we find%
\begin{multline*}
-b_{2}(n+2)K_{\ k,r}^{r}+B(n+2)E_{\ kr}^{r}-B(n+2)E_{k\ r}^{\ r}-b_{2}nT_{\
kr}^{r}+b_{2}T_{k\ r}^{\ r}- \\
2(n+2)(n-1)Y_{k}=0,
\end{multline*}%
whence the second equation results. Finally, the third one follows from
Lemma \ref{AfterLE5}.
\end{proof}

\begin{lemma}
Under assumptions of Lemma \ref{Case3 Lemma 2} suppose $L_{X}g=0.$ Then%
\begin{equation*}
AY=BY=A^{\prime }Y=0
\end{equation*}%
at each point $(x,0)\in TM.$
\end{lemma}

\begin{proof}
By (\ref{20}), $Y$ is a Killing vector field on $M.$ Moreover, (\ref{L5b})
reduces to 
\begin{equation*}
2A^{\prime }g_{kl}Y_{p}+B\left( Y_{k}g_{lp}+Y_{l}g_{kp}\right) =0,
\end{equation*}%
whence we easily deduce $BY=A^{\prime }Y=0.$ Since an infinitesimal isometry
is also an infinitesimal affine transformation, from (\ref{L6b}), by the use
of (\ref{Conv5}) and the above properties, we obtain $AY=0.$
\end{proof}

\begin{lemma}
Under assumptions of Lemma \ref{Case3 Lemma 2} suppose 
\begin{equation}
P+L_{X}g=0.  \label{Case3-8-1}
\end{equation}
Then $\nabla P=0$ if and only if $BY=A^{\prime }Y=0.$
\end{lemma}

\begin{proof}
Substituting into (\ref{L5b}), symmetrizing in $(k,p)$ and applying (\ref%
{Case3-8-1}) we get 
\begin{equation*}
a_{2}\nabla _{l}P_{kp}+B(2g_{kp}Y_{l}+g_{lp}Y_{k}+g_{kl}Y_{p})+2A^{\prime
}(g_{kl}Y_{p}+g_{lp}Y_{k})=0,
\end{equation*}%
whence the thesis results.
\end{proof}

A complete lift of a Killing vector field on $M$ to $(TM,G)$ is always a
Killing vector field (see the next section). Thus we have proved

\begin{theorem}
\label{Structure 2}Let on $TM,$ $dimTM>4,$ a $g-$ natural metric $G$%
\begin{eqnarray*}
G_{(x,u)}(X^{h},Y^{h}) &=&A(r^{2})g_{x}(X,Y)+B(r^{2})g_{x}(X,u)g_{x}(Y,u), \\
G_{(x,u)}(X^{h},Y^{v})
&=&a_{2}(r^{2})g_{x}(X,Y)+b_{2}(r^{2})g_{x}(X,u)g_{x}(Y,u), \\
G_{(x,u)}(X^{v},Y^{h})
&=&a_{2}(r^{2})g_{x}(X,Y)+b_{2}(r^{2})g_{x}(X,u)g_{x}(Y,u), \\
G_{(x,u)}(X^{v},Y^{v}) &=&0
\end{eqnarray*}%
be given, where $a_{2}b_{2}\neq 0$ everywhere on $TM$ while $%
b_{2}-a_{2}^{\prime }$ and either $A$ or $B$ do not vanish on a dense subset
of $TM.$ If $Z$ is a Killing vector field on $TM,$ then there exists an open
subset $U$ containing $M$ such that $Z$ restricted to $U$ is a complete lift
of a Killing vector field on $M,$ i.e.%
\begin{equation*}
Z_{|U}=X^{C}.
\end{equation*}
\end{theorem}

\subsection{Case 4}

The class under consideration contains the Sasaki metric $g^{S}$ and the
Cheeger-Gromol one $g^{CG}.$ In (\cite{Tanno 2}) Tanno proved the following

\begin{theorem}
Let $(M,g)$ be a Riemannian manifold. Let $X$ be a Killing vector field on $%
M,$ $P$ be a $(1,1)$ tensor field on $M$ that is skew-symmetric and parallel
and $Y$ be a vector field on $M$ that satisfies $\nabla _{k}\nabla
_{l}Y_{p}+\nabla _{l}\nabla _{k}Y_{p}=0$ and (\ref{LemmaCase4-3-1}). Then
the vector field $Z$ on $TM$ defined by%
\begin{equation*}
Z=X^{C}+\iota P+Y^{\#}=(X^{r}-\nabla ^{r}Y_{s}u^{s})\partial
_{r}^{h}+(Y^{r}+S_{s}^{r}u^{s})\partial _{r}^{v}
\end{equation*}%
is a Killing vector field on $(TM,g^{S}).$ Conversely, any Killing vector
field on $(TM,g^{S})$ is of this form.
\end{theorem}

A similar theorem holds for $(TM,g^{CG}),$ (\cite{Abbassi 2003}). However,
in virtue of Lemma \ref{AfterLE5} and the remark after it, the $Y$ component
vanishes.

We shall give a simple sufficient condition for $\iota P$ to be a Killing
vector field on $TM.$ The rest of the section is devoted to investigations
on the properties of the $Y$ component.

Notice that $a\neq 0$ and $a_{2}=0$ require $a_{1}A\neq 0.$ From (\ref{L5b})
we get immediately%
\begin{equation}
a_{1}\left( \nabla _{k}\nabla _{l}Y_{p}+\nabla _{l}\nabla _{k}Y_{p}\right)
=2A^{\prime }g_{kl}Y_{p}+B\left( Y_{k}g_{lp}+Y_{l}g_{kp}\right) .
\label{LEC4-0-1}
\end{equation}

Since $b_{2}=0,$ symmetrizing (\ref{II2}) in $\left( k,p\right) $ we get $%
AE_{lkp}=a_{2}^{\prime }(g_{lk}Y_{p}+g_{lp}Y_{k}).$ Consequently$,$ in
virtue of the properties of the Lie derivative (\ref{Conv5}), (\ref{L6b}) \
and (\ref{22}) yield 
\begin{equation*}
a_{1}\nabla _{l}P_{kp}=a_{2}^{\prime }(g_{lp}Y_{k}-g_{lk}Y_{p}).
\end{equation*}

Moreover, because of $a_{2}=0,$ $b_{2}=0,$ $\overline{S}_{pq}=0$ and $\nabla
_{l}X_{q}+S_{lq}=P_{lq}=-P_{ql},$ identity (\ref{I3}) together with (\ref%
{L5a})\ yields 
\begin{equation*}
BP_{kp}=a_{2}^{\prime }\nabla _{k}Y_{p},
\end{equation*}%
whence, since $P$ is skew-symmetric,%
\begin{equation}
a_{2}^{\prime }L_{Y}g=0\text{ and }a_{2}^{\prime }Tr(\nabla Y)=0
\label{LEC4-0-3}
\end{equation}%
result.

Next, Lemma \ref{LE8} yields%
\begin{equation*}
B\nabla _{k}X_{m}-a_{2}^{\prime }\nabla _{k}Y_{m}+BP_{km}=0,
\end{equation*}%
whence we find%
\begin{equation*}
B\nabla X=0.
\end{equation*}%
We conclude with

\begin{lemma}
Suppose (\ref{H}), $dimM>2,$ and $a_{2}=0,$ $b_{2}=0$ on $M\times \{0\}.$\
If $a_{2}^{\prime }=0$ on $M\times \{0\},$ then $BP=0$ and $\nabla P=0$ on $%
M\times \{0\}.$
\end{lemma}

By Proposition \ref{jotaP} we obtain

\begin{theorem}
Suppose $a_{2}(r^{2})=0,$ $b_{2}(r^{2})=0$ and $B(r^{2})=0$ on $(TM,G).$ If $%
M$ admits non-trivial skew-symmetric and parallel $(0,2)$ tensor field $P,$
then its $\iota $-lift is a Killing vector field on $TM.$
\end{theorem}

\begin{lemma}
\label{LE12}If $a_{2}=0,$ $b_{2}=0$ at $(x,0)$ and $a\neq 0$ everywhere on $%
TM,$ then 
\begin{multline}
3a_{1}^{2}\left( \nabla ^{r}Y_{q}R_{rlkp}+\nabla ^{r}Y_{p}R_{rlkq}\right) =
\label{LE12-1} \\
a_{1}B\left[ \left( 2\nabla _{q}Y_{k}-\nabla _{k}Y_{q}\right) g_{pl}+\left(
2\nabla _{p}Y_{k}-\nabla _{k}Y_{p}\right) g_{ql}-\left( \nabla
_{p}Y_{q}+\nabla _{q}Y_{p}\right) g_{kl}\right] + \\
2A(b_{1}-a_{1}^{\prime })\left( 2\nabla _{l}Y_{k}g_{pq}-\nabla
_{l}Y_{p}g_{kq}-\nabla _{l}Y_{q}g_{kp}\right)
\end{multline}%
and%
\begin{multline*}
3a_{1}^{2}\nabla ^{q}Y_{p}R_{qlkr}u^{p}u^{r}= \\
2A(b_{1}-a_{1}^{\prime })\left[ Y_{k,l}r^{2}-Y_{p,l}u_{k}u^{p}\right] +a_{1}B%
\left[ \left( 2Y_{k,p}-Y_{p,k}\right) u_{l}-g_{kl}Y_{p,q}u^{q}\right] u^{p}
\end{multline*}%
hold at arbitrary point $(x,0)\in TM.$
\end{lemma}

\begin{proof}
To prove the lemma it is enough to put $a_{2}=0,$ $b_{2}=0$ in (\ref{like 78}%
), then multiply by $A$ and apply (\ref{22}). For convenience indices $(k,m)$
are interchanged after that.
\end{proof}

\begin{lemma}
\label{LemmaCase4-3}Suppose (\ref{H}), $dimM>2,$ $a\neq 0.$ If neither $%
\nabla _{n}Y_{m}=0$ nor $\nabla _{n}Y_{m}=\frac{T}{n}g_{mn},$ then 
\begin{equation}
\nabla ^{r}Y_{q}R_{rlkp}+\nabla ^{r}Y_{p}R_{rlkq}=0  \label{LemmaCase4-3-1}
\end{equation}%
if and only if $B=b=0.$
\end{lemma}

\begin{proof}
The "only if " part is obvious. Put $T=Y_{r,s}g^{rs}.$ Suppose that (\ref%
{LemmaCase4-3-1}) holds. Contracting the right hand side of (\ref{LE12-1}) \
in turn with $g^{kl},$ $g^{kl}g^{mn},$ $g^{lm}$ and $g^{kn}$ we get
respectively%
\begin{equation*}
\left[ 2Ab+(n-1)a_{1}B\right] \left( Y_{m,n}+Y_{n,m}\right) -4AbTg_{mn}=0,
\end{equation*}%
\begin{equation*}
(n-1)\left( 2Ab-a_{1}B\right) T=0,
\end{equation*}%
\begin{equation*}
-\left[ 4Ab+(2n+1)a_{1}B\right] Y_{k,n}+\left[ 2Ab+(n+2)a_{1}B\right]
Y_{n,k}+2AbTg_{kn}=0,
\end{equation*}%
\begin{equation*}
-a_{1}BY_{l,m}+2\left[ (n-1)Ab+a_{1}B\right] Y_{m,l}-a_{1}BTg_{lm}=0.
\end{equation*}%
If $Y_{l,m}-Y_{m,l}\neq 0,$ then alternating the last two equations in
indices we obtain 
\begin{eqnarray*}
2Ab+(n+1)a_{1}B &=&0, \\
2(n-1)Ab+3a_{1}B &=&0,
\end{eqnarray*}%
whence $B=b=0$ for $n\neq 2$ results.

If $Y_{l,m}-Y_{m,l}=0,$ then the suitable linear combination of these
equations gives%
\begin{equation*}
(n-2)\left( 2Ab-a_{1}B\right) Y_{l,m}+\left( 2Ab-a_{1}B\right) Tg_{lm}=0.
\end{equation*}%
By the second equation this yields $\left( 2Ab-a_{1}B\right) Y_{l,m}=0.$
Applying the last result to the first equality completes the proof.
\end{proof}

\begin{lemma}
\label{LemmaCase4-6}\label{LE13}Let $(TM,G)$ be a tangent bundle of a
manifold $(M,g),$ $dimM>2,$ with $g$-natural metric $G$ given by (\ref{g1a}%
). Suppose there is given a Killing vector field $Z$ on $TM$ with Taylor
expansion (\ref{Taylor1}). If the coefficients $a_{2}(t),$ $b_{2}(t)$ vanish
along $M$ then $Y$ satisfies%
\begin{equation}
A^{\prime }(2b_{1}+a_{1}^{\prime })Y=0,  \label{LC4-1}
\end{equation}%
\begin{equation}
\left\{ \left[ 2B(B+A^{\prime })-3AB^{\prime }\right]
a_{1}+AB(2b_{1}+a_{1}^{\prime })\right\} Y=0.  \label{LC4-2}
\end{equation}
\end{lemma}

\begin{proof}
Recall that if $a_{2}(r_{0}^{2})=0,$ then necessary $a_{1}A\neq 0$ on some
neighbourhood of $r_{0}^{2}.$

From (\ref{L5c}) we easily get%
\begin{equation}
A(\overline{K}_{ab,c}+\overline{K}_{bc,a}+\overline{K}_{ca,b})+2(B+A^{\prime
})(g_{bc}Y_{a}+g_{ca}Y_{b}+g_{ab}Y_{c})=0.  \label{LC4-4}
\end{equation}%
From Lemma \ref{LE6'}, by the use of the assumptions on $a_{2}$ and $b_{2},$
we find%
\begin{multline*}
2B\left[ \overline{K}_{ab,(k}g_{l)c}+\overline{K}_{bc,(k}g_{l)a}+\overline{K}%
_{ca,(k}g_{l)b}\right] + \\
3B\left[ g_{c(l}T_{k)ab}+g_{a(l}T_{k)bc}+g_{b(l}T_{k)ca}\right] + \\
6A^{\prime }g_{kl}M_{abc}-b\left[ g_{ab}\left( Y_{c,kl}+Y_{c,lk}\right)
+g_{bc}\left( Y_{a,kl}+Y_{a,lk}\right) +g_{ca}\left(
Y_{b,kl}+Y_{b,lk}\right) \right] + \\
6B^{\prime }\left[ \left( g_{al}g_{kb}+g_{ak}g_{lb}\right) Y_{c}+\left(
g_{bl}g_{kc}+g_{bk}g_{lc}\right) Y_{a}+\left(
g_{al}g_{kc}+g_{ak}g_{lc}\right) Y_{b}\right] =0.
\end{multline*}%
Applying (\ref{L5b}), (\ref{L5a2}) and (\ref{L5a2'}) we find%
\begin{multline}
2B\left[ \overline{K}_{ab,(k}g_{l)c}+\overline{K}_{bc,(k}g_{l)a}+\overline{K}%
_{ca,(k}g_{l)b}\right] -  \label{LC4-5} \\
\frac{4bB}{a_{1}}\left[ \left( g_{al}g_{bc}+g_{bl}g_{ca}+g_{cl}g_{ab}\right)
Y_{k}+\left( g_{ak}g_{bc}+g_{bk}g_{ca}+g_{ck}g_{ab}\right) Y_{l}\right] - \\
\frac{4A^{\prime }(2b_{1}+a_{1}^{\prime })}{a_{1}}\left(
g_{ab}Y_{c}+g_{bc}Y_{a}+g_{ca}Y_{b}\right) g_{kl}+\frac{6(B^{\prime
}a_{1}-Ba_{1}^{\prime })}{a_{1}}\times \\
\left[ \left( g_{al}g_{kb}+g_{ak}g_{lb}\right) Y_{c}+\left(
g_{bl}g_{kc}+g_{bk}g_{lc}\right) Y_{a}+\left(
g_{al}g_{kc}+g_{ak}g_{lc}\right) Y_{b}\right] =0.
\end{multline}%
Hence, contracting with $g^{kl},$ we obtain%
\begin{multline}
B(\overline{K}_{ab,c}+\overline{K}_{bc,a}+\overline{K}_{ca,b})+
\label{LC4-6} \\
\left[ 3B^{\prime }-\frac{(B+nA^{\prime })(2b_{1}+a_{1}^{\prime })}{a_{1}}%
\right] (g_{bc}Y_{a}+g_{ca}Y_{b}+g_{ab}Y_{c})=0.
\end{multline}%
If $B\neq 0,$ then a linear combination of (\ref{LC4-4}) and (\ref{LC4-6})
yields $\psi Y=0$ where%
\begin{equation}
\psi =2B(B+A^{\prime })-3AB^{\prime }+\frac{A(B+nA^{\prime
})(2b_{1}+a_{1}^{\prime })}{a_{1}}.  \label{LC4-7}
\end{equation}%
On the other hand, contractions of (\ref{LC4-5}) with $g^{ak}$ and then with$%
\ g^{bl}$ yield respectively%
\begin{multline*}
B\left[ (n+3)\overline{K}_{bc,l}+\overline{K}_{c,r}^{r}g_{bl}+\overline{K}%
_{b,r}^{r}g_{cl}\right] = \\
\frac{2}{a_{1}}\left[ (n+3)bB+A^{\prime }(2b_{1}+a_{1}^{\prime })\right]
g_{bc}Y_{l}+ \\
\frac{1}{a_{1}}\left[ 2bB+3(n+2)(Ba_{1}^{\prime }-B^{\prime
}a_{1})+2A^{\prime }(2b_{1}+a_{1}^{\prime })\right] (g_{bl}Y_{c}+g_{cl}Y_{b})
\end{multline*}%
and 
\begin{equation*}
2B\overline{K}_{c,r}^{r}=\frac{1}{a_{1}}\left[ 4bB+3(n+1)(Ba_{1}^{\prime
}-B^{\prime }a_{1})+2A^{\prime }(2b_{1}+a_{1}^{\prime })\right] Y_{c}.
\end{equation*}%
Hence we find%
\begin{multline}
2(n+3)a_{1}B\overline{K}_{bc,l}=4\left[ (n+3)bB+A^{\prime
}(2b_{1}+a_{1}^{\prime })\right] g_{bc}Y_{l}+  \label{LC4-9} \\
\left[ 3(n+3)(Ba_{1}^{\prime }-B^{\prime }a_{1})+2A^{\prime
}(2b_{1}+a_{1}^{\prime })\right] (g_{bl}Y_{c}+g_{cl}Y_{b})
\end{multline}%
and%
\begin{multline*}
(n+3)a_{1}B(\overline{K}_{bc,l}+\overline{K}_{cl,b}+\overline{K}_{lb,c})- \\
\left[ (n+3)(B(2b_{1}+a_{1}^{\prime })-3a_{1}B^{\prime })+4A^{\prime
}(2b_{1}+a_{1}^{\prime })\right] (g_{bc}Y_{l}+g_{cl}Y_{b}+g_{lb}Y_{c})=0
\end{multline*}%
If $B\neq 0,$ then combining the last relation with (\ref{LC4-6}) we obtain (%
\ref{LC4-1}) and, as a consequence of (\ref{LC4-7}), equality (\ref{LC4-2}).
On the other hand, if $B(r_{0}^{2})=0,$ then contractions of (\ref{LC4-9})
with $g^{bc}$ and $g^{bl}$ yield either $Y^{a}=0$ or $B^{\prime }=0$ and $%
A^{\prime }(2b_{1}+a_{1}^{\prime })=0.$ This completes the proof.
\end{proof}

\begin{lemma}
For an arbitrary $B$ we have%
\begin{equation*}
a_{1}^{2}B(\nabla _{l}\nabla _{c}Y_{b}+\nabla _{l}\nabla
_{b}Y_{c})=-2ABbg_{bc}Y_{l}-\frac{3}{2}A(Ba_{1}^{\prime }-a_{1}B^{\prime
})(g_{bl}Y_{c}+g_{cl}Y_{b}),
\end{equation*}%
\begin{multline*}
a_{1}^{2}B(\nabla _{l}\nabla _{c}Y_{b}-\nabla _{b}\nabla _{l}Y_{c})=-B\left(
a_{1}B+2Ab\right) g_{bc}Y_{l}+ \\
-\dfrac{1}{2}\left[ 4A^{\prime }a_{1}B+3A(Ba_{1}^{\prime }-a_{1}B^{\prime })%
\right] g_{bl}Y_{c}-\dfrac{1}{2}\left[ 2a_{1}B^{2}+3A(Ba_{1}^{\prime
}-a_{1}B^{\prime })\right] g_{cl}Y_{b}.
\end{multline*}
\end{lemma}

\begin{proof}
We can suppose $B\neq 0.$ From (\ref{II1}) we get 
\begin{equation*}
A\nabla _{l}\overline{K}_{km}=-a_{1}\left( \nabla _{l}\nabla
_{k}Y_{m}+\nabla _{l}\nabla _{m}Y_{k}\right) .
\end{equation*}%
Combining this with (\ref{LC4-9}), by the use of (\ref{LC4-1}) we find the
first equality. Hence, by the use of (\ref{L5b}) and (\ref{LC4-1}), we get
the second one. On the other hand, if $BY=0,$ then the previous lemma yields 
$B^{\prime }Y=0.$ This completes the proof.
\end{proof}

\begin{lemma}
\label{LEC4-10}Under hypothesis (\ref{H}) suppose $\dim M>2,$ $a\neq 0$ on $%
TM$ and $a_{2}=0,$ $b_{2}=0$ on $M\times \{0\}\subset TM.$ Then 
\begin{equation}
\left[ Aa_{2}^{\prime }(b_{1}+a_{1}^{\prime })-2a_{1}(Ba_{2}^{\prime
}+Ab_{2}^{\prime })\right] Y=0,  \label{LEC4-10-1}
\end{equation}%
\begin{equation}
Y\left[ a_{1}a_{2}^{\prime }R-\frac{(Ba_{2}^{\prime }+2Ab_{2}^{\prime })}{2}%
g\wedge g\right] =0.  \label{LEC4-10-2}
\end{equation}%
If $a_{2}^{\prime }\neq 0,$ then%
\begin{equation}
b_{2}^{\prime }\nabla Y=0,\quad (b_{1}-a_{1}^{\prime })\nabla Y=0
\label{LEC4-10-2-1}
\end{equation}
\end{lemma}

\begin{proof}
For the proof of the first part we apply Lemma \ref{LE6''}. Substituting $%
a_{2}=0,$ $b_{2}=0,$ by the use of (\ref{22}), we get%
\begin{multline*}
a_{1}\left[
2E_{ab}^{p}R_{plck}-E_{bk}^{p}R_{plac}+E_{bc}^{p}R_{plak}-E_{ak}^{p}R_{plbc}+E_{ac}^{p}R_{plbk}%
\right] + \\
B\left[ \left( E_{ckb}-E_{kcb}\right) g_{al}+\left( E_{cak}-E_{kac}\right)
g_{bl}+\right. \\
\left. \left( E_{abk}+E_{bak}\right) g_{cl}-\left( E_{abc}+E_{bac}\right)
g_{kl}\right] -2a_{2}^{\prime }\left( M_{abk}g_{cl}-M_{abc}g_{kl}\right) + \\
2b_{2}^{\prime }\left[ \left( g_{bk}g_{cl}-g_{bc}g_{kl}\right) Y_{a}+\left(
g_{ak}g_{cl}-g_{ac}g_{kl}\right) Y_{b}+\right. \\
\left. \left( g_{al}g_{bk}+g_{ak}g_{bl}\right) Y_{c}-\left(
g_{al}g_{bc}+g_{ac}g_{bl}\right) Y_{k}\right] =0.
\end{multline*}%
Applying (\ref{L5a1}) - (\ref{L5a2'}) and the Bianchi identity we find%
\begin{multline*}
-a_{1}^{2}a_{2}^{\prime }\left[
3R_{blck}Y_{a}+3R_{alck}Y_{b}+(R_{akbl}+R_{bkal})Y_{c}-(R_{acbl}+R_{bcal})Y_{k}%
\right] - \\
2a_{2}^{\prime }\left[ A(b_{1}+a_{1}^{\prime })-a_{1}B\right] g_{ab}\left(
g_{kl}Y_{c}-g_{cl}Y_{k}\right) - \\
\left[ 2Aa_{2}^{\prime }(b_{1}+a_{1}^{\prime })+a_{1}(2Ab_{2}^{\prime
}-Ba_{2}^{\prime })\right] \left[ \left( g_{bc}g_{kl}-g_{bk}g_{cl}\right)
Y_{a}+\left( g_{ac}g_{kl}-g_{ak}g_{cl}\right) Y_{b}\right] + \\
a_{1}(Ba_{2}^{\prime }+2Ab_{2}^{\prime })\left[ g_{bl}\left(
g_{ak}Y_{c}-g_{ac}Y_{k}\right) +g_{al}\left( g_{bk}Y_{c}-g_{bc}Y_{k}\right) %
\right] =0.
\end{multline*}%
Symmetrizing the last relation in $(a,b,l)$ we obtain (\ref{LEC4-10-1}).
Then, symmetrize in $(a,b,k).$ Since coefficient times $Y_{c}$ vanishes by (%
\ref{LEC4-10-1}), by the use of the the Walker lemma and (\ref{LEC4-10-1})
we get either $Y=0$ or 
\begin{equation*}
a_{1}a_{2}^{\prime }(R_{acbl}+R_{albc})=(Ba_{2}^{\prime }+2Ab_{2}^{\prime
})(g_{al}g_{bc}+g_{ac}g_{bl}-2g_{ab}g_{cl}),
\end{equation*}%
whence, alternating in $(b,l),$ we easily obtain%
\begin{equation*}
a_{1}a_{2}^{\prime }R_{acbl}=(Ba_{2}^{\prime }+2Ab_{2}^{\prime
})(g_{al}g_{bc}-g_{ab}g_{cl}).
\end{equation*}%
Thus (\ref{LEC4-10-2}) is proved.

Suppose now $Y\neq 0$ and $a_{2}^{\prime }\neq 0$ on $M\times \{0\}.$
Applying (\ref{LEC4-10-2}) to (\ref{LE12-1}) and eliminating $B,$ by the use
of (\ref{LEC4-0-3}), we obtain%
\begin{multline}
2(b_{1}-a_{1}^{\prime })\left(
Y_{n,l}g_{km}+Y_{m,l}g_{kn}-2Y_{k,l}g_{mn}\right) +  \label{LEC4-10-6} \\
\left( 2\frac{a_{1}b_{2}^{\prime }}{a_{2}^{\prime }}-b_{1}-a_{1}^{\prime
}\right) (Y_{k,m}g_{ln}+Y_{k,n}g_{lm})- \\
\left( 4\frac{a_{1}b_{2}^{\prime }}{a_{2}^{\prime }}+b_{1}+a_{1}^{\prime
}\right) (Y_{m,k}g_{ln}+Y_{n,k}g_{lm})=0.
\end{multline}%
Contracting (\ref{LEC4-10-6}) with $g^{lm},$ by the use of (\ref{LEC4-0-3})
we get%
\begin{equation*}
\left[ (n+1)\frac{a_{1}b_{2}^{\prime }}{a_{2}^{\prime }}-(b_{1}-a_{1}^{%
\prime })\right] Y_{k,n}=0.
\end{equation*}%
On the other hand, by contraction with $g^{mn}$ we obtain%
\begin{equation*}
\left[ \frac{a_{1}b_{2}^{\prime }}{a_{2}^{\prime }}-(n-1)(b_{1}-a_{1}^{%
\prime })\right] Y_{k,l}=0.
\end{equation*}%
Hence we easily get either $\nabla Y=0$ or both $b_{2}^{\prime }=0$ and $%
b_{1}-a_{1}^{\prime }=0$ on $M\times \{0\}.$

\begin{remark}
If $a_{2}^{\prime }\neq 0$ and $Y\neq 0,$ then equations (\ref{LEC4-10-2-1})
give a further restriction on the metric $G.$ Namely, if $\nabla Y=0,$ then
from (\ref{22}) we infer $K=0$ while from (\ref{LEC4-0-1}) we get $A^{\prime
}=0$ and $B=0$ on $M\times \{0\}.$ Consequently, (\ref{LC4-2}) yields $%
B^{\prime }=0.$

On the other hand, substituting $b_{2}^{\prime }=0$ and $b_{1}=a_{1}^{\prime
}$ into Lemma \ref{AfterLE5} we get $b_{1}^{\prime }=0$ on $M\times \{0\}.$
\end{remark}
\end{proof}

\section{Lifts properties}

\subsection{$X^{C}$ and $X^{v}$}

If $X=X^{r}\partial _{r}$ is a vector field on $M,$ then $%
X^{C}=X^{r}\partial _{r}+u^{s}\partial _{s}X^{r}\delta _{r}=X^{r}\partial
_{r}^{h}+u^{s}\nabla _{s}X^{r}\partial _{r}^{v}$ is said to be the complete
lift of $X$ to $TM.$

\begin{lemma}
Let $X$ be a vector field on $(M,g)$ satisfying%
\begin{equation}
L_{X}g=fg,  \label{Kvf}
\end{equation}%
$f$ being a function on $M,$ and $X^{C}$ be its complete lift to $(TM,G)$
with $g$-natural metric $G.$ Then%
\begin{eqnarray*}
\left( L_{X^{C}}G\right) \left( \partial _{k}^{h},\partial _{l}^{h}\right)
&=&\left[ a_{2}\partial f+f(A+A^{\prime }r^{2})\right] g_{kl}+f(2B+B^{\prime
}r^{2})u_{k}u_{l}+ \\
&&\frac{1}{2}b_{2}r^{2}\left( \nabla _{k}fu_{l}+\nabla _{l}fu_{k}\right) ,
\end{eqnarray*}%
\begin{eqnarray*}
\left( L_{X^{C}}G\right) \left( \partial _{k}^{v},\partial _{l}^{h}\right)
&=&\frac{1}{2}a_{1}\left( \nabla _{l}fu_{k}-\nabla _{k}u_{l}+\partial
fg_{kl}\right) + \\
&&f(a_{2}+a_{2}^{\prime }r^{2})g_{kl}+f(2b_{2}+b_{2}^{\prime
}r^{2})u_{k}u_{l}+\frac{1}{2}b_{1}r^{2}\nabla _{l}fu_{k},
\end{eqnarray*}%
\begin{equation*}
\left( L_{X^{C}}G\right) \left( \partial _{k}^{v},\partial _{l}^{v}\right)
=f(a_{1}+a_{1}^{\prime }r^{2})g_{kl}+f(2b_{1}+b_{1}^{\prime
}r^{2})u_{k}u_{l},
\end{equation*}%
where $\partial f=u^{r}\nabla _{r}f.$
\end{lemma}

\begin{proof}
Straightforward calculations with the use of (\ref{LD10}) - (\ref{LD12}) .
Relations (\ref{Conv3}) and (\ref{Conv5}) are useful.
\end{proof}

\begin{theorem}
\label{Lift prop 2}Let $X$ be a vector field on $(M,g)$ such that (\ref{Kvf}%
) is satisfied. Then $X^{C}$ is a Killing vector field on $(TM,G)$ with
non-degenerated $g$-natural metric $G$ if and only if $f=0$ on $M.$
\end{theorem}

\begin{proof}
If $f=0,$ then the theorem is obvious by the previous lemma.

Suppose that $L_{X^{C}}G=0$ on $TM$ holds for some $f\neq 0.$ At first,
contracting the third equation with $g^{kl},$ next transvecting with $%
u^{k}u^{l},$ we easily find 
\begin{gather*}
f(a_{1}+a_{1}^{\prime }r^{2})=0, \\
f(2b_{1}+b_{1}^{\prime }r^{2})=0.
\end{gather*}%
Hence we obtain $a_{1}=0$ and $b_{1}=0$ on $TM$ since the only smooth
solution to $h(r^{2})+r^{2}h^{\prime }(r^{2})=0$ that can be prolonged
smoothly on \thinspace $<0,\infty )$ is $h(r^{2})=0.$ Then the second
equation gives $a_{2}=0$ on $TM,$ a contradiction to $a(r^{2})\neq 0.$ This
completes the proof.
\end{proof}

\begin{proposition}
The vertical lift $X^{v}=X^{r}\partial _{r}^{v}$ of a Killing vector field $%
X=X^{r}\partial _{r}$ to $(TM,G)$ with $g-$natural metrics $G$ is a Killing
vector field on $TM$ if and only if $a_{j}^{\prime }=0$ and $b_{j}=0$ on $%
TM. $
\end{proposition}

\begin{proof}
Suppose $X^{v}$ is a Killing vector field. Since $X$ is also the Killing
one, (\ref{LD10}) yields%
\begin{equation*}
b_{2}(X_{r,k}u_{l}+X_{r,l}u_{k})u^{r}+B(u_{k}X_{l}+u_{l}X_{k})+2u^{r}X_{r}(A^{\prime }g_{kl}+B^{\prime }u_{k}u_{l})=0,
\end{equation*}%
whence, by contraction with $g^{kl}$ and $u^{k}u^{l}$ we obtain 
\begin{equation*}
2u^{r}X_{r}(B+nA^{\prime }+r^{2}B^{\prime })=0
\end{equation*}%
and%
\begin{equation*}
2r^{2}u^{r}X_{r}(B+A^{\prime }+r^{2}B^{\prime })=0
\end{equation*}%
since $X$ is a Killing vector field on $M.$ Thus $A^{\prime }=0$ and the
only smooth solution to $B+r^{2}B^{\prime }=0$ on $TM$ is $B=0$. In similar
manner, from (\ref{LD11}) and (\ref{LD12}) we deduce that $a_{1}^{\prime
}=a_{2}^{\prime }=0$ and $b_{1}=b_{2}=0$ on $TM.$ The "only if" part is
obvious. Thus the proposition is proved.
\end{proof}

\subsection{$V^{a}\partial _{a}^{v}=u^{p}\protect\nabla ^{r}Y_{p}\partial
_{r}^{v}$}

Let $Y$ be a non-parallel Killing vector field on $M$ and consider its lift $%
u^{p}\nabla ^{r}Y_{p}\partial _{r}^{v}$ to $(TM,G).$Then we have $\partial
_{k}^{v}V^{a}=\nabla ^{a}Y_{k},$ $\partial _{k}^{h}V^{a}=u^{p}\Theta
_{k}\left( \nabla ^{a}Y_{p}\right) $ and from (\ref{LD10}) - (\ref{LD12}) we
obtain

\begin{align*}
\left( L_{V^{a}\partial _{a}^{v}}G\right) \left( \partial _{k}^{h},\partial
_{l}^{h}\right) & =a_{2}(\nabla _{l}\nabla _{k}Y_{p}+\nabla _{l}\nabla
_{k}Y_{p})u^{p}+B(\nabla _{k}Y_{p}u^{p}u_{l}+\nabla _{l}Y_{p}u^{p}u_{k}), \\
\left( L_{V^{a}\partial _{a}^{v}}G\right) \left( \partial _{k}^{v},\partial
_{l}^{h}\right) & =a_{2}\nabla _{l}Y_{k}+a_{1}\nabla _{l}\nabla
_{k}Y_{p}u^{p}+b_{2}\nabla _{l}Y_{p}u^{p}u_{k}, \\
\left( L_{V^{a}\partial _{a}^{v}}G\right) \left( \partial _{k}^{v},\partial
_{l}^{v}\right) & =0.
\end{align*}%
Hence we deduce

\begin{proposition}
Let $Y$ be a non-parallel Killing vector field on $M$ satisfying $\nabla
\nabla Y=0.$ Then $u^{p}\nabla ^{r}Y_{p}\partial _{r}^{v}$ is a Killing
vector field on $TM$ if and only if $a_{2}=b_{2}=B=0$ on $TM.$
\end{proposition}

\begin{proposition}
Let $Y$ be a non-parallel Killing vector field on $M.$ If $a_{2}=b_{2}=B=0$
on $TM$ and $u^{p}\nabla ^{r}Y_{p}\partial _{r}^{v}$ is a Killing vector
field on $TM$ then $\nabla \nabla Y=0$ on $M$.
\end{proposition}

\begin{proof}
It is enough to symmetrize the second equation.
\end{proof}

\subsection{$\protect\iota P$}

\begin{proposition}
\label{jotaP}Let $P$ be an arbitrary (0,2)-tensor field on $(M,g).$ The its $%
\iota $ - lift $\iota P=u^{r}P_{r}^{a}\partial _{a}^{v}$ to $\left(
TM,G\right) $ with $g-$natural metric $G$ satisfies%
\begin{multline*}
\left( L_{\iota P}G\right) \left( \partial _{k}^{h},\partial _{l}^{h}\right)
=a_{2}u^{r}\left( \nabla _{k}P_{lr}+\nabla _{l}P_{kr}\right)
+b_{2}u^{p}u^{r}\left( \nabla _{k}P_{pr}u_{l}+\nabla _{l}P_{pr}u_{k}\right) +
\\
2(A^{\prime }g_{kl}+B^{\prime }u_{k}u_{l})P_{pr}u^{p}u^{r}+Bu^{r}\left(
P_{kr}u_{l}+P_{lr}u_{k}\right) , \\
\left( L_{\iota P}G\right) \left( \partial _{k}^{v},\partial _{l}^{h}\right)
=a_{2}P_{lk}+b_{2}u^{r}P_{rk}u_{l}+a_{1}u^{r}\nabla _{l}P_{kr}+b_{1}\nabla
_{l}P_{r}^{a}u_{a}u^{r}u_{k}+ \\
2(a_{2}^{\prime }g_{kl}+b_{2}^{\prime
}u_{k}u_{l})P_{pr}u^{p}u^{r}+b_{2}u^{r}\left( P_{kr}u_{l}+P_{lr}u_{k}\right)
, \\
\left( L_{\iota P}G\right) \left( \partial _{k}^{v},\partial _{l}^{v}\right)
=a_{1}\left( P_{kl}+P_{lk}\right) +b_{1}\left[ u^{p}\left(
P_{kp}+P_{pk}\right) u_{l}+u^{p}\left( P_{lp}+P_{pl}\right) u_{k}\right] + \\
2(a_{1}^{\prime }g_{kl}+b_{1}^{\prime }u_{k}u_{l})P_{pr}u^{p}u^{r}.
\end{multline*}
\end{proposition}

\begin{proof}
We have $V^{a}=u^{r}P_{r}^{a},$ $\partial _{k}^{v}V^{a}=P_{k}^{a},$ $%
\partial _{k}^{h}V^{a}=u^{p}\left( \partial _{k}P_{p}^{a}-\Gamma
_{pk}^{t}P_{t}^{a}\right) $ and $\partial _{k}^{h}V^{a}+V^{r}\Gamma
_{kr}^{a}=u^{p}\nabla _{k}P_{p}^{a}.$
\end{proof}

\begin{proposition}
Let $P$ be a skew-symmetric (0,2)-tensor field on $(M,g).$ The its $\iota $
- lift $\iota P=u^{r}P_{r}^{a}\partial _{a}^{v}$ to $\left( TM,G\right) $
with $g-$natural metric $G$ satisfies%
\begin{multline*}
\left( L_{\iota P}G\right) \left( \partial _{k}^{h},\partial _{l}^{h}\right)
=a_{2}\left( u^{r}\nabla _{k}P_{lr}+u^{r}\nabla _{l}P_{kr}\right) +B\left(
u^{r}P_{kr}u_{l}+u^{r}P_{lr}u_{k}\right) , \\
\left( L_{\iota P}G\right) \left( \partial _{k}^{v},\partial _{l}^{h}\right)
=a_{2}P_{lk}+b_{2}u^{r}P_{lr}u_{k}+a_{1}u^{r}\nabla _{l}P_{kr}, \\
\left( L_{\iota P}G\right) \left( \partial _{k}^{v},\partial _{l}^{v}\right)
=0.
\end{multline*}
\end{proposition}

\begin{proof}
By definition, $\iota P=u^{r}P_{r}^{t}\partial _{t}^{v},$ where $%
P_{r}^{t}=P_{xr}g^{xt}.$ Thus it is enough to check the identities making
use of (\ref{LD10}) - (\ref{LD12}) with $H^{a}=0,$ $V^{a}=u^{r}P_{r}^{a},$ $%
V_{k}=-u^{r}P_{rk}.$
\end{proof}

\subsubsection{$\protect\iota C^{\left[ X\right] }$}

Put $C^{\left[ X\right] }=\left( \left( C^{\left[ X\right] }\right)
_{k}^{h}\right) =\left( -g^{hr}\left( L_{X}g\right) _{rk}\right) =\left(
-\left( \nabla ^{h}X_{k}+\nabla _{k}X^{h}\right) \right) $ on $(M,g)$ Then
its $\iota $-lift $\iota C^{\left[ X\right] }=\left( 0,\ u^{k}\left( C^{%
\left[ X\right] }\right) _{k}^{h}\right) =\left( 0,\ -u^{k}\left( \nabla
^{h}X_{k}+\nabla _{k}X^{h}\right) \right) \ $is a vertical vector field on $%
TM.$ In adapted coordinates $\left( \partial _{k}^{v},\partial
_{l}^{h}\right) $ we have%
\begin{equation*}
\iota C^{\left[ X\right] }=-u^{k}\left( \nabla ^{h}X_{k}+\nabla
_{k}X^{h}\right) \partial _{h}^{v}.
\end{equation*}%
Hence, applying Lemma \ref{Lie Deriv}, we easily get%
\begin{multline*}
\left( L_{\iota C^{\left[ X\right] }}G\right) \left( \partial
_{k}^{h},\partial _{l}^{h}\right) = \\
-a_{2}u^{p}\left[ \nabla _{k}\nabla _{l}X_{p}+\nabla _{l}\nabla
_{k}X_{p}+\nabla _{k}\nabla _{p}X_{l}+\nabla _{l}\nabla _{p}X_{k}\right] - \\
2b_{2}u^{p}u^{q}\left[ \nabla _{k}\nabla _{p}X_{q}u_{l}+\nabla _{l}\nabla
_{p}X_{q}u_{k}\right] -4\left( A^{\prime }g_{kl}+B^{\prime
}u_{k}u_{l}\right) u^{p}u^{q}\nabla _{p}X_{q}- \\
Bu^{p}\left[ \left( \nabla _{k}X_{p}+\nabla _{p}X_{k}\right) u_{l}+\left(
\nabla _{l}X_{p}+\nabla _{p}X_{l}\right) u_{k}\right] ,
\end{multline*}%
\begin{multline*}
\left( L_{\iota C^{\left[ X\right] }}G\right) \left( \partial
_{k}^{v},\partial _{l}^{h}\right) = \\
-a_{2}\left( \nabla _{k}X_{l}+\nabla _{l}X_{k}\right) -b_{2}u^{p}\left[
2\left( \nabla _{k}X_{p}+\nabla _{p}X_{k}\right) u_{l}+\left( \nabla
_{l}X_{p}+\nabla _{p}X_{l}\right) u_{k}\right] - \\
a_{1}u^{p}\left( \nabla _{l}\nabla _{k}X_{p}+\nabla _{l}\nabla
_{p}X_{k}\right) -2b_{1}u^{p}u^{q}\nabla _{l}\nabla _{p}X_{q}u_{k}- \\
4\left( a_{2}^{\prime }g_{kl}+b_{2}^{\prime }u_{k}u_{l}\right)
u^{p}u^{q}\nabla _{p}X_{q},
\end{multline*}%
\begin{multline*}
\left( L_{\iota C^{\left[ X\right] }}G\right) \left( \partial
_{k}^{v},\partial _{l}^{v}\right) = \\
-2a_{1}\left( \nabla _{k}X_{l}+\nabla _{l}X_{k}\right) -2b_{1}u^{p}\left[
\left( \nabla _{k}X_{p}+\nabla _{p}X_{k}\right) u_{l}+\left( \nabla
_{l}X_{p}+\nabla _{p}X_{l}\right) u_{k}\right] - \\
4\left( a_{1}^{\prime }g_{kl}+b_{1}^{\prime }u_{k}u_{l}\right)
u^{p}u^{q}\nabla _{p}X_{q}.
\end{multline*}

\subsubsection{Complete lift X$^{C}$ of X to (TM, G)}

We have $X^{C}=(X^{r}\partial _{r})^{C}=X^{r}\partial _{r}+\partial
X^{r}\delta _{r}=X^{r}\partial _{r}^{h}+u^{p}\nabla _{p}X^{r}\partial
_{r}^{v}.$ Making use of Lemma \ref{Lie Deriv} we obtain%
\begin{multline*}
\left( L_{X^{C}}G\right) \left( \partial _{k}^{h},\partial _{l}^{h}\right) =
\\
a_{2}u^{p}\left[ \nabla _{k}\nabla _{p}X_{l}+X^{r}R_{rkpl}+\nabla _{l}\nabla
_{p}X_{k}+X^{r}R_{rlpk}\right] + \\
b_{2}u^{p}u^{q}\left[ \left( \nabla _{k}\nabla
_{p}X_{q}+X^{r}R_{rkpq}\right) u_{l}+\left( \nabla _{l}\nabla
_{p}X_{q}+X^{r}R_{rlpq}\right) u_{k}\right] + \\
A\left( \nabla _{k}X_{l}+\nabla _{l}X_{k}\right) +Bu^{p}\left[ \left( \nabla
_{k}X_{p}+\nabla _{p}X_{k}\right) u_{l}+\left( \nabla _{l}X_{p}+\nabla
_{p}X_{l}\right) u_{k}\right] + \\
2\left( A^{\prime }g_{kl}+B^{\prime }u_{k}u_{l}\right) u^{p}u^{q}\nabla
_{p}X_{q},
\end{multline*}%
\begin{multline*}
\left( L_{X^{C}}G\right) \left( \partial _{k}^{v},\partial _{l}^{h}\right) =
\\
a_{1}u^{p}\left[ \nabla _{l}\nabla _{p}X_{k}+X^{r}R_{rlpk}\right]
+a_{2}\left( \nabla _{k}X_{l}+\nabla _{l}X_{k}\right) + \\
b_{2}u^{p}\left[ \left( \nabla _{k}X_{p}+\nabla _{p}X_{k}\right)
u_{l}+\left( \nabla _{l}X_{p}+\nabla _{p}X_{l}\right) u_{k}\right] + \\
b_{1}u^{p}u^{q}\left( \nabla _{l}\nabla _{p}X_{q}+X^{r}R_{rlpq}\right)
u_{k}+2\left( a_{2}^{\prime }g_{kl}+b_{2}^{\prime }u_{k}u_{l}\right)
u^{p}u^{q}\nabla _{p}X_{q},
\end{multline*}%
\begin{multline*}
\left( L_{X^{C}}G\right) \left( \partial _{k}^{v},\partial _{l}^{v}\right) =
\\
a_{1}\left( \nabla _{k}X_{l}+\nabla _{l}X_{k}\right) +b_{1}u^{p}\left[
\left( \nabla _{k}X_{p}+\nabla _{p}X_{k}\right) u_{l}+\left( \nabla
_{l}X_{p}+\nabla _{p}X_{l}\right) u_{k}\right] + \\
2\left( a_{1}^{\prime }g_{kl}+b_{1}^{\prime }u_{k}u_{l}\right)
u^{p}u^{q}\nabla _{p}X_{q}.
\end{multline*}

\subsubsection{\label{Subsection543}$\protect\iota C^{\left[ X\right]
}+X^{C} $ for an infinitesimal affine transformation}

Suppose that $X$ is an infinitesimal affine transformation on $M.$ Then by (%
\ref{Conv5}) and the definition

\begin{equation*}
\nabla _{k}\nabla _{l}X_{p}+\nabla _{k}\nabla _{p}X_{l}=\nabla _{k}\nabla
_{l}X_{p}+X^{r}R_{rklp}+\nabla _{k}\nabla _{p}X_{l}+X^{r}R_{rkpl}=0
\end{equation*}%
and 
\begin{equation*}
u^{p}u^{q}\nabla _{k}\nabla _{p}X_{q}=-u^{p}u^{q}X^{r}R_{rkpq}=0.
\end{equation*}%
Therefore, applying results of previous subsections, we find%
\begin{equation*}
\left( L_{\iota C^{\left[ X\right] }+X^{C}}G\right) \left( \partial
_{k}^{h},\partial _{l}^{h}\right) =A\left( \nabla _{k}X_{l}+\nabla
_{l}X_{k}\right) -2\left( A^{\prime }g_{kl}+B^{\prime }u_{k}u_{l}\right)
u^{p}u^{q}\nabla _{p}X_{q},
\end{equation*}%
\begin{multline*}
\left( L_{\iota C^{\left[ X\right] }+X^{C}}G\right) \left( \partial
_{k}^{v},\partial _{l}^{h}\right) = \\
-2\left( a_{2}^{\prime }g_{kl}+b_{2}^{\prime }u_{k}u_{l}\right)
u^{p}u^{q}\nabla _{p}X_{q}-b_{2}u^{p}\left( \nabla _{k}X_{p}+\nabla
_{p}X_{k}\right) u_{l},
\end{multline*}%
\begin{multline*}
\left( L_{\iota C^{\left[ X\right] }+X^{C}}G\right) \left( \partial
_{k}^{v},\partial _{l}^{v}\right) =-a_{1}\left( \nabla _{k}X_{l}+\nabla
_{l}X_{k}\right) - \\
b_{1}u^{p}\left[ \left( \nabla _{k}X_{p}+\nabla _{p}X_{k}\right)
u_{l}+\left( \nabla _{l}X_{p}+\nabla _{p}X_{l}\right) u_{k}\right] -2\left(
a_{1}^{\prime }g_{kl}+b_{1}^{\prime }u_{k}u_{l}\right) u^{p}u^{q}\nabla
_{p}X_{q}.
\end{multline*}

\section{Appendix{}}

The Levi-Civita connection $\widetilde{\nabla }$ of the $g$ - natural metric 
$G$ on $TM$ was calculated and presented in (\cite{Abb 2005 a}, \cite{Abb
2005 b}, \cite{Abb 2005 c}). Unfortunately, they contain some misprints and
omissions. Below, we present the correct version owing to the courtesy of
the authors of the aforementioned papers. In addition, it was checked
independently by the present author.

Moreover, observe that it is the Levi-Civita connection of a metric given by
(\ref{g1a}).

Let $T$ be a tensor field of type\thinspace $(1,s)$ on $M.$ For any $%
X_{1},...,X_{s}\in T_{x}M,$ $x\in M,$ we define horizontal and vertical
vectors at a point $(x,u)\in TTM$ setting respectively%
\begin{equation*}
h\left\{ T(X_{1},...,u,...,X_{s-1}\right\}
=\sum_{r=1}^{dimM}u^{r}[T(X_{1},...,\partial _{r},..,X_{s-1})]^{h},
\end{equation*}%
\begin{equation*}
v\left\{ T(X_{1},...,u,...,X_{s-1}\right\}
=\sum_{r=1}^{dimM}u^{r}[T(X_{1},...,\partial _{r},..,X_{s-1})]^{v}.
\end{equation*}

By the similar formulas we define 
\begin{equation*}
h\left\{ T(X_{1},...,u,...,u,...,X_{s-1}\right\} \ \text{and}\ h\left\{
T(X_{1},...,u,...,u,...,X_{s-1}\right\} .
\end{equation*}
Moreover, we put $h\left\{ T(X_{1},...,X_{s}\right\} =\left(
T(X_{1},...,X_{s})\right) ^{h}$ and $v\left\{ T(X_{1},...,X_{s}\right\}
=\left( T(X_{1},...,X_{s})\right) ^{v}.$ Therefore $h\{X\}=X^{h}$ and $%
v\{X\}=X^{v}$ (\cite{Abb 2005 a}, p. 22-23).

Finally, we write \thinspace 
\begin{equation*}
R(X,Y,Z)=R(X,Y)Z\ \text{and}\ R(X,Y,Z,V)=g(R(X,Y,Z),V)
\end{equation*}
for all $X,Y,Z,V\in T_{x}M.$

\begin{proposition}
\label{Connection}(\cite{Abb 2005 a}, \cite{Abb 2005 b}, \cite{Abb 2005 c})
Let $(M,g)$ be a Riemannian manifold, $\nabla $ its Levi-Civita connection
and $R$ its Riemann curvature tensor. If $G$ is a $g-$natural metrics on $%
TM, $ then the Levi-Civita connection $\widetilde{\nabla }$of $(TM,G)$ at a
point $(x,u)\in TM$ is given by 
\begin{eqnarray*}
\left( \widetilde{\nabla }_{X^{h}}Y^{h}\right) _{(x,u)} &=&\left( \nabla
_{X}Y\right) _{(x,u)}^{h}+h\left\{ A(u,X_{x},Y_{x})\right\} +v\left\{
B(u,X_{x},Y_{x})\right\} , \\
\left( \widetilde{\nabla }_{X^{h}}Y^{v}\right) _{(x,u)} &=&\left( \nabla
_{X}Y\right) _{(x,u)}^{v}+h\left\{ C(u,X_{x},Y_{x})\right\} +v\left\{
D(u,X_{x},Y_{x})\right\} , \\
\left( \widetilde{\nabla }_{X^{v}}Y^{h}\right) _{(x,u)} &=&h\left\{
C(u,Y_{x},X_{x})\right\} +v\left\{ D(u,Y_{x},X_{x})\right\} , \\
\left( \widetilde{\nabla }_{X^{v}}Y^{v}\right) _{(x,u)} &=&h\left\{
E(u,X_{x},Y_{x})\right\} +v\left\{ F(u,X_{x},Y_{x})\right\} ,
\end{eqnarray*}%
for all vector fields $X,$ $Y$ on $M,$ where $P=a_{2}^{\prime }-\frac{b_{2}}{%
2},$ $Q=$ $a_{2}^{\prime }+\frac{b_{2}}{2}$ and%
\begin{multline*}
A(u,X,Y)=-\frac{a_{1}a_{2}}{2a}\left[ R(X,u,Y)+R(Y,u,X)\right] + \\
\frac{a_{2}B}{2a}\left[ g(Y,u)X+g(X,u)Y\right] + \\
\frac{1}{aF}\QATOPD\{ . {{}}{{}}\left. a_{2}\left[ a_{1}\left(
F_{1}B-F_{2}b_{2}\right) +a_{2}\left( b_{1}a_{2}-b_{2}a_{1}\right) \right]
R(X,u,Y,u)+\right. \\
\left. \left[ aF_{2}B^{\prime }+B\left[
a_{2}(F_{2}b_{2}-F_{1}B)+A(a_{1}b_{2}-a_{2}b_{1})\right] \right]
g(X,u)g(Y,u)+\right. \\
\left. aF_{2}A^{\prime }g(X,Y)\right. \QATOPD. \} {{}}{{}}u,
\end{multline*}%
\begin{multline*}
B(u,X,Y)=\frac{a_{2}^{2}}{a}R(X,u,Y)-\frac{a_{1}A}{2a}R(X,Y,u)- \\
\frac{AB}{2a}\left[ g(Y,u)X+g(X,u)Y\right] + \\
\frac{1}{aF}\QATOPD\{ .
{{}}{{}}a_{2}[a_{2}(F_{2}b_{2}-F_{1}B)+A(b_{2}a_{1}-b_{1}a_{2})]R(X,u,Y,u)+
\\
\left[ \frac{{}}{{}}\right. -a(F_{1}+F_{3})B^{\prime }+B\left[ A\left(
(F_{1}+F_{3})b_{1}-F_{2}b_{2}\right) +a_{2}\left( a_{2}B-b_{2}A\right) %
\right] \left. \frac{{}}{{}}\right] g(X,u)g(Y,u) \\
-a(F_{1}+F_{3})A^{\prime }g(X,Y)\QATOPD. \} {{}}{{}}u,
\end{multline*}%
\begin{multline*}
C(u,X,Y)=-\frac{a_{1}^{2}}{2a}R(Y,u,X)+\frac{a_{1}B}{2a}g(X,u)Y+ \\
\frac{1}{a}\left( a_{1}A^{\prime }-a_{2}P\right) g(Y,u)X+ \\
\frac{1}{aF}\QATOPD\{ . {{}}{{}}\frac{a_{1}}{2}\left[ a_{2}\left(
a_{2}b_{1}-a_{1}b_{2}\right) +a_{1}\left( F_{1}B-F_{2}b_{2}\right) \right]
R(X,u,Y,u)+ \\
a\left( \frac{F_{1}}{2}B+F_{2}P\right) g(X,Y)+ \\
\QATOPD[ . {{}}{{}}aF_{1}B^{\prime }+\left( A^{\prime }+\frac{B}{2}\right) %
\left[ a_{2}\left( a_{1}b_{2}-a_{2}b_{1}\right) +a_{1}\left(
F_{2}b_{2}-BF_{1}\right) \right] + \\
P\left[ a_{2}\left( b_{1}\left( F_{1}+F_{3}\right) -b_{2}F_{2}\right)
-a_{1}\left( b_{2}A-a_{2}B\right) \right] \QATOPD. ]
{{}}{{}}g(X,u)g(Y,u)\QATOPD. \} {{}}{{}}u,
\end{multline*}%
\begin{multline*}
D(u,X,Y)=\frac{1}{a}\left\{ \QATOP{{}}{{}}\frac{a_{1}a_{2}}{2}R(Y,u,X)-\frac{%
a_{2}B}{2}g(X,u)Y\right. + \\
\left. \left( AP-a_{2}A^{\prime })g(Y,u)X\right) \QATOP{{}}{{}}\right\} + \\
\frac{1}{aF}\left\{ \frac{a_{1}}{2}\left[
A(a_{1}b_{2}-a_{2}b_{1})+a_{2}(F_{2}b_{2}-F_{1}B)\right] R(X,u,Y,u)\right. -
\\
a\left[ \frac{F_{2}}{2}B+(F_{1}+F_{3})P\right] g(X,Y)+ \\
\left[ \QATOP{{}}{{}}-aF_{2}B^{\prime }+\left( A^{\prime }+\frac{B}{2}%
\right) \left[ A(a_{2}b_{1}-a_{1}b_{2})+a_{2}(F_{1}B-F_{2}b_{2})\right]
+\right. \\
\left. \left. P\left[ A(b_{2}F_{2}-b_{1}(F_{1}+F_{3}))+a_{2}(b_{2}A-a_{2}B)%
\right] \QATOP{{}}{{}}\right] g(X,u)g(Y,u)\right\} u,
\end{multline*}%
\begin{multline*}
E(u,X,Y)=\frac{1}{a}\left( a_{1}Q-a_{2}a_{1}^{\prime }\right) \left[
g(X,u)Y+g(Y,u)X\right] + \\
\frac{1}{aF}\left\{ \QATOP{{}}{{}}a\left[ F_{1}b_{2}-F_{2}(b_{1}-a_{1}^{%
\prime })\right] g(X,Y)+\right. \\
\left[ \QATOP{{}}{{}}a(2F_{1}b_{2}^{\prime }-F_{2}b_{1}^{\prime
})+2a_{1}^{\prime }\left[
a_{1}(a_{2}B-b_{2}A)+a_{2}(b_{1}(F_{1}+F_{3})-b_{2}F_{2})\right] +\right. \\
\left. \left. 2Q\left[ a_{1}(F_{2}b_{2}-F_{1}B)+a_{2}(a_{1}b_{2}-a_{2}b_{1})%
\right] \QATOP{{}}{{}}\right] g(X,u)g(Y,u)\QATOP{{}}{{}}\right\} u,
\end{multline*}%
\begin{multline*}
F(u,X,Y)=\frac{1}{a}\left( Aa_{1}^{\prime }-a_{2}Q\right) \left[
g(X,u)Y+g(Y,u)X\right] + \\
\frac{1}{aF}\left\{ \QATOP{{}}{{}}a\left[ (F_{1}+F_{3})(b_{1}-a_{1}^{\prime
})-F_{2}b_{2}\right] g(X,Y)+\right. \\
\left[ \QATOP{{}}{{}}a((F_{1}+F_{3})b_{1}^{\prime }-2F_{2}b_{2}^{\prime
})+\right. 2a_{1}^{\prime }\left[
a_{2}(b_{2}A-a_{2}B)+A(b_{2}F_{2}-b_{1}(F_{1}+F_{3}))\right] + \\
\left. \left. 2Q\left[ a_{2}(F_{1}B-F_{2}b_{2})+A(a_{2}b_{1}-a_{1}b_{2})%
\right] \QATOP{{}}{{}}\right] g(X,u)g(Y,u)\right\} u.
\end{multline*}
\end{proposition}

A $(0,4)$ tensor $B$ on a manifold $M$ is said to be a generalized curvature
tensor if%
\begin{equation*}
B(V,X,Y,Z)+B(V,Y,Z,X)+B(V,Z,X,Y)=0
\end{equation*}%
and%
\begin{equation*}
B(V,X,Y,Z)=-B(X,V,Y,Z),\qquad B(V,X,Y,Z)=B(Y,Z,V,X)
\end{equation*}%
for all vector fields $V,$ $X,$ $Y,$ $Z$ on $M$ (\cite{Nomizu}).

For a $(0,k)$ tensor $T,$ $k\geq 1,$ we define 
\begin{equation*}
(R\cdot T)(X_{1},...,X_{k};X,Y)=\nabla _{Y}\nabla
_{X}T(X_{1},...,X_{k})-\nabla _{X}\nabla _{Y}T(X_{1},...,X_{k}).
\end{equation*}%
For more details see for example (\cite{Belkhelfa}), (\cite{Ewert}).

The Kulkarni-Nomizu product of symmetric $(0,2)$ tensors $A$ and $B$ is
given by 
\begin{multline*}
\left( A\wedge B\right) (U,X,Y,V)= \\
A(X,Y)B(U,V)-A(X,V)B(U,Y)+A(U,V)B(X,Y)-A(U,Y)B(X,V).
\end{multline*}

\begin{theorem}
\label{Apen1}\cite{Grycak} Let $(M,g)$ be a semi-Riemannian manifold with
metric $g$, $dimM>2.$ Let $g_{X}$ be a 1-form associated to $g,$ i.e. $%
g_{X}(Y)=g(Y,X)$ for any vector field $Y.$

If $B$ is generalized curvature tensor having the property $R\cdot B=0$ and $%
P$ is a one-form on $M$ satisfying%
\begin{equation}
\left( R\cdot V\right) \left( X;Y,Z\right) =\left( P\wedge g_{X}\right)
\left( Y,Z\right) ,  \label{Ap1}
\end{equation}%
for some 1-form $V,$ then%
\begin{equation*}
P\left( B-\frac{TrB}{2n(n-1)}g\wedge g\right) =0.
\end{equation*}

If $A$ is a symmetric $(0,2)$-tensor on $M$ having the properties $R\cdot
A=0 $ and (\ref{Ap1}) then 
\begin{equation*}
P\left( A-\frac{TrA}{n}g\right) =0.
\end{equation*}
\end{theorem}

\begin{lemma}
\cite{Walker}\label{Apen2} Let $A_{l},$ $B_{hk}$ where $l,h,k=1,...,n$ be
numbers satisfying 
\begin{equation*}
B_{hk}=B_{kh},\quad A_{l}B_{hk}+A_{h}B_{kl}+A_{k}B_{lh}=0.
\end{equation*}%
Then either $A_{l}=0$ for all $l$ or $B_{hk}=0$ for all $h,k.$
\end{lemma}

\begin{lemma}
\label{Apen3}Let on a manifold $(M,g),$ $dimM>2,$ $(0,2)$ tensors $A,$ $B,$ $%
F$ satisfying the condition%
\begin{multline*}
g(X,Y)F(U,V)+g(U,V)B(X,Y)+ \\
g(Y,V)A(X,U)+g(X,V)A(Y,U)+g(Y,U)A(X,V)+g(X,U)A(Y,V)=0
\end{multline*}%
for arbitrary vectors $X,Y,U,V$ be given.

Then $F$ and $B$ are symmetric. Moreover, $A=0$, $B+F=0$ and $%
nF-TrFg=nB-TrBg=0.$
\end{lemma}

\begin{proof}
In local coordinates $(U,(x^{a}))$ the condition writes%
\begin{equation*}
g_{kl}F_{mn}+g_{mn}B_{kl}+g_{ln}A_{km}+g_{kn}A_{lm}+g_{lm}A_{kn}+g_{km}A_{ln}=0.
\end{equation*}%
By contractions with $g^{kl},$ $g^{mn},$ $g^{km}$ we obtain in turn%
\begin{eqnarray*}
2(A_{mn}+A_{nm})+nF_{mn}+B_{p}^{p}g_{mn} &=&0, \\
2(A_{kl}+A_{lk})+nB_{kl}+F_{p}^{p}g_{kl} &=&0, \\
(n+2)A_{ln}+B_{nl}+F_{ln}+A_{p}^{p}g_{ln} &=&0.
\end{eqnarray*}%
Now, the symmetry of $F$ and $B$ results from the first two equations.
Contracting the first equation with $g^{mn}$ and the third one with $g^{ln}$
we get%
\begin{eqnarray*}
4A_{p}^{p}+n(B_{p}^{p}+F_{p}^{p}) &=&0, \\
2(n+1)A_{p}^{p}+B_{p}^{p}+F_{p}^{p} &=&0,
\end{eqnarray*}%
whence $TrA=TrF+TrB=0$ results. Applying these to the first system we easily
get%
\begin{eqnarray*}
4(A_{mn}+A_{nm})+n(F_{mn}+B_{mn}) &=&0, \\
(n+2)(A_{mn}+A_{nm})+2(F_{mn}+B_{mn}) &=&0,
\end{eqnarray*}%
whence $F+B=0$ and $A_{mn}+A_{nm}=0.$ Now the third equation yields $A=0.$
The further statements are obvious.
\end{proof}

Stanis\l aw Ewert-Krzemieniewski

West Pomeranian University of Technology Szczecin

School of Mathematics

Al. Piast\'{o}w 17

70-310 Szczecin

Poland

e-mail: ewert@zut.edu.pl

\end{document}